\documentclass[a4paper,11pt]{article}

\usepackage{xcolor}
\usepackage{enumitem}
\usepackage{graphicx}
\usepackage{leftindex}
\usepackage[verbose]{hyperref}
\usepackage{tikz}
\usepackage[margin=1.in]{geometry} 
\usepackage{float}  
\usepackage{amsmath, amsthm, amsfonts}
\usepackage{authblk}

\newcommand{\cls}[1]{\overline{#1}}
\newcommand{\nablaLCC}{\hat \nabla}
\newcommand{\withboundaryM}{{M}}
\newcommand{\withboundaryA}{{A}}
\newcommand{\withboundaryS}{{S}}
\newcommand{\entorno}[1]{\widehat{#1}}
\newcommand{\functionspace}{\mathcal{F}}
\newcommand{\R}{\mathbb{R}}
\newcommand{\Lo}{\mathbb{L}}
\newcommand{\SSS}{\mathbb{S}}
\newcommand{\Ng}{\mathcal{N}} 
\newcommand{\nablaM}{\nabla}
\newcommand{\nablai}[1]{{\leftindex^{#1}\nabla}}
\newcommand{\second}[1]{\leftindex^{#1}II}
\newcommand{\lc}[1]{\mathcal{C}_{#1}}
\newcommand{\tmv}[1]{\mathbb{T}_{#1}}

\newtheorem{theorem}{Theorem}[section]
\newtheorem{proposition}[theorem]{Proposition} \newtheorem{lemma}[theorem]{Lemma}
\newtheorem{corollary}[theorem]{Corollary} 
\newtheorem{definition}[theorem]{Definition}
\newtheorem{exe}[theorem]{Example}

\newtheorem{remark}[theorem]{Remark}
 \newtheorem{convention}[theorem]{Convention}

\usetikzlibrary{decorations.markings}
\usetikzlibrary{arrows.meta}
\tikzset{
  pics/carc/.style args={#1:#2:#3}{
    code={
      \draw[pic actions] (#1:#3) arc(#1:#2:#3);
    }
  }
}
\tikzset{mylabel/.style  args={at #1 #2  with #3}{
    postaction={decorate,
    decoration={
      markings,
      mark= at position #1
      with  \node [#2] {#3};
 } } } }
\begin{document}

\markboth{J. Herrera and M. Sánchez}{Finslerian lightconvex boundaries: applications to causal simplicity and the space of cone geodesics $\Ng$}


\author[1]{J\'onatan Herrera\footnote{Corresponding author. Email: jherrera@uco.es}}
\affil[1]{Departamento de Matemáticas, Edificio Albert Einstein, Universidad de Córdoba, Campus de Rabanales, 14071, Córdoba, Spain}
\author[2]{Miguel S\'anchez}
\affil[2]{Departamento de Geometría y Topología, Facultad de Ciencias, Universidad de Granada, Campus Fuentenueva S/N, 18071, Granada, Spain}

\title{Finslerian lightconvex boundaries: \\ applications to causal simplicity and the space of cone geodesics $\Ng$}

\date{}

\maketitle

\begin{abstract} 
Our  outcome is structured in the following sequence: (1) 
a general result for indefinite Finslerian manifolds with boundary $(M,L)$  showing the equivalence between local and infinitesimal (time, light or space) convexities for the boundary $\partial M$, (2) for any cone structure $(M,\lc{} )$ which is globally hyperbolic with timelike boundary, the equivalence among: (a) the boundary $\partial M$ is lightconvex, (b) the interior $\mathring{M}$ is causally simple and (c) the space of the cone (null) geodesics $\Ng$ of $(\mathring{M},\lc{} )$ is Hausdorff,  (3) in this case, the manifold structure of $\Ng$ is obtained explicitly in terms of elements in $\partial M$ and a smooth Cauchy hypersurface $S$, (4) the known results and examples about  Hausdorfness of $\Ng$  are revisited and extended, leading to the notion of {\em causally simple spacetime with $T_2$-lightspace} as a step in the causal ladder  below global hyperbolicity.

The results are   significant for   relativistic (Lorentz)  spacetimes and the writing allows one 
either  to be introduced in  Finslerian technicalities or to skip them. In particular,  asymptotically AdS spacetimes  become examples where the $C^{1,1}$ conformal extensions at infinity yield totally lightgeodesic boundaries, and all the results above apply.  
\end{abstract}

\newpage

\tableofcontents
\newpage

\section{Introduction}

We analyze the relations among some basic ingredients related to 
a general (Lorentz-Finsler) cone structure $\lc{}$ on a manifold $M$ with a timelike boundary $\partial M$: 
the space of its cone (or null) 
geodesics  $\Ng$,   
the lightconvexity of $\partial M$, and its step  (causally simple, globally hyperbolic) in the  causal ladder.  


The space of null geodesics $\Ng$, or  {\em lightspace} for short,  was introduced   by R. Low \cite{Low_1989} (in the Lorentz case); remarkable achievements include  the solution by Chernov and Nemirovski \cite{Chernov_2010} of Low's conjecture on causal linking and the possibility to define  new types of causal boundaries   
\cite{Bautista_2014, Bautista_2022, Chernov_2020, Low_2006}. 
Low's  foundational properties included the existence of a differentiable structure on $\Ng$ when $\lc{}$ is strongly causal. For globally hyperbolic spacetimes,
$\Ng$  also becomes Hausdorff and, then, a standard manifold, identifiable to  one constructed from any smooth Cauchy hypersurface. Low  also emphasized the  elusiveness of   Hausdorffness and posed its characterization  as an open problem (see \cite[pp. 3008, 3015]{Low_2001}). Hedicke and Suhr \cite{Hedicke_2019} exhibited an example of causally simple spacetime with non-Hausdorff $\Ng$ (disproving a conjecture by Chernov \cite{Chernov_2018});    
other related properties on causally simple spacetimes (including affine {\em pseudoconvexity} in \cite{BEE} and the convexity of Riemannian manifolds in \cite{BGS}) where studied by Hedicke et al. \cite{Hedicke-2021-contact, Hedicke_2021}. 

Here, our aim is to show that, in the class of globally hyperbolic cones  $\lc{}$ with a timelike boundary $\partial M$, the causal simplicity of the interior $\mathring{M}$ and the Hausdorffness of the lightspace $\Ng$ of $\mathring{M}$ are equivalent. Moreover, when these properties hold, $\Ng$ can be described explicitly. From the technical viewpoint, the key of the equivalence comes from the property of lightconvexity for $\partial M$, an issue with interest in its own right in the general (semi-)Finslerian setting.  

 Recently, globally hyperbolic (Lorentz) spacetimes-with-timelike-boundary have been studied systematically in \cite{RMI_AFS}. This is a very general class of spacetimes where $\partial M$ may represent either the restriction to a bounded region of a (possibly globally hyperbolic) spacetime without boundary or the set of naked singularities, eventually lying at infinity (as in the case of AdS spacetime). 
Among the properties proved in \cite{RMI_AFS}, it is essential here a global splitting  of the spacetime with its boundary, extending the one in the case without boundary proved by Bernal and one of the authors \cite{BS05}. Indeed, this splitting permits to determine the global structure of $\Ng$ now. 
All the properties to be used here can be extended to the Lorentz-Finsler case \cite[Theorem 3.3]{Sa-BIRS} (see also \cite{Sa_Penrose}
). The detailed Lorentz-Finsler framework developed in \cite{JS20} will be used here. 

The generality of the  class of spacetimes   admitting a timelike boundary, in addition to other results and  examples in \S \ref{s_last}, suggests to consider the class of causally simple $\lc{}$ with $T_2$-lightspace as a new step in the causal ladder of (Finsler) spacetimes.

Next,  a more detailed explanation by sections is given.


\smallskip

\noindent
Section \ref{s2} analyzes the notion of convexity in the general  semi-Finsler setting (i.e.,  permitting any non-degenerate  signature for the fundamental tensor) and introduces briefly the required Finsler background. This includes the notion of anisotropic connection (developed in the Appendix), which permits a treatment  not too far from the standard Lorentzian (or semi-Riemannian) setting.

The relation between the geodesics in the interior  $\mathring{M}$ of the manifold  and the properties of its boundary $\partial M$, is an issue related to classical notions on convexity for $\mathring{M}$ (no geodesic with endpoints in $\mathring{M}$ can touch $\partial M$) and for $\partial M$ (its inner pointing second fundamental form is positive semidefinite). In the Riemannian case, a very geometric proof was obtained by Bishop \cite{Bishop}, under $C^4$ smoothness. The proof could not be extended to the Finslerian case, but one of the authors and his coworkers \cite{Bartolo_2010} gave a more analytic one to both, include this case and optimize regularity to $C^{1,1}$. Hintz and Uhlmann  \cite{Hintz_2017} solved a similar  question for light convexity in Lorentz manifolds. Taking into account these two  previous approaches, here we give a general proof for semi-Finsler metrics. It is worth pointing out that this covers the general abstract case of {\em causal structures of signature $(k+1,l+1)$} developed by O. Makhmali in \cite{Makh} (see below Remark \ref{r_sign_kl}).

Summing up, as a direct consequence of 
 Theorems  \ref{prop:main2:6}, \ref{prop:main2:4}  and Proposition \ref{prop_strictconv}, we obtain:

\begin{theorem}\label{t1}
    For any semi-Finsler manifold $(M^{n+1},L,A)$ with non-degenerate boundary $\partial M$ they are equivalent:
\begin{enumerate}[label=(\roman*)]
    \item $\partial M$ is infinitesimally (resp. light, time or space) convex, i.e., the second fundamental form $\second{\eta}$ with respect to any inwards vector field $\eta$ transverse to $\partial M$  (rigging vector) is positive semidefinite for the vectors of the corresponding type in $T\, \partial M$.

    \item $\partial M$ is locally (resp. light, time or space)
 convex, i.e., any (resp. lightlike,  timelike, spacelike) geodesic $  \gamma: [0,b]\rightarrow \withboundaryM$ with $\gamma(0)=p\in\partial M$ and $\dot{\gamma}(0)\in T_p\partial M$ is contained in $\partial M$.
\end{enumerate}
\end{theorem}

\smallskip

\noindent Section \ref{s3} focuses on the relation between  causal simplicity for $\mathring{M}$ and  lightconvexity for $\partial M$, which will be closely related to the Hausdorffness of the lightspace $\Ng$. The necessary elements on cone structures and causality in manifolds with boundary are recalled. An essential ingredient of the proof of our main result (Theorem \ref{thm:main2:4}) is the previous equivalence of lightconvexities. 

The following two subtleties   are also worth mentioning. 
First, even in the Lorentz case with $C^\infty$-boundary $\partial M$, the  horismotic (and causal)  relation must be computed by lowering the piecewise smooth regularity of the curves \cite[Section 2.4, Appendix B]{RMI_AFS}. Consistently, locally Lipschitz curves $H^1$-parametrized will be used here (see around footnote \ref{f_RMI}). 
Second, in the Finslerian setting one has two natural choices of  anisotropic connections: (i) the Berwald anisotropic connection, which is very simple (as it is computed only from the geodesic spray) and widely used along this paper, and (ii) the Levi-Civita-Chern anisotropic one which  is more entangled (see the Appendix) but it is required at this point of our proof. Of course, in the Lorentz setting both anisotropic connections match the (linear) Levi-Civita connection, and the reader interested only in this case  can skip this. 

As a consequence of Theorem \ref{thm:main2:4} and the preliminary results on the lightspace   in \S \ref{s_4} (Propositions \ref{p_noHaus}, \ref{p_Haus}), we obtain:

\begin{theorem}\label{t2}
For any globally hyperbolic cone structure $\lc{}$ on a manifold-with-timelike-boundary $M^{n+1}$ the following properties are equivalent:

\begin{enumerate}
    \item The boundary $\partial M$ is lightconvex. 
    \item The interior $\mathring{M}$ is causally simple. 
    \item The space of cone  geodesics $\Ng$ of $\mathring{M}$ is  Hausdorff. 
\end{enumerate}
\end{theorem}

\smallskip \noindent
Section \ref{s_4} studies  the  lightspace (space of lightlike or {\em cone} geodesics for any $\lc{}$)  $\Ng$  in detail. Subsection~\ref{s_4.1} revisits Low's approach \cite{Low_1989, Low_2001, Low_2006} and proves the aforementioned results on Hausdorffness for the globally hyperbolic case with boundary. Subsection~\ref{s_4.2} goes beyond by showing  that, in this class, 
the global structure of   $\Ng$ as a manifold can be determined explicitly  (see Theorem~\ref{t_Ng_estructuraglobal}):

\begin{quote} {\em When the equivalent properties in Theorem \ref{t2} hold,   the differentiable structure of $\Ng$ as a smooth Hausdorff $(2n-1)$-manifold (without boundary), can be described explicitly in terms of  three elements: the projectivization of the tangent bundle $TS$ of any spacelike  Cauchy hypersurface  $S$ in $M^{n+1}$, the set of inwards lightdirections starting at points in $J^+(S)\cap \partial M$ and the set of  outwards lightdirections  starting at points in $J^-(S)\cap \partial M$. }
\end{quote}
This description can be simplified in some cases, as the following one (Corollary \ref{cor:partialM}):
\begin{quote}
    {\em In particular, if every cone geodesic has a future  endpoint at $\partial M$, then $\Ng$ is diffeomorphic to $\R\times T \, \partial S$.}
\end{quote}

In Subsection \ref{s_4.3},  all the previous results are applied to anti de Sitter (AdS) spacetime, and generalized to spacetimes whose global conformal structures behave as  AdS at infinity, i.e., satisfying a weak  condition on AdS asymptoticity.
It is worth pointing out that the structure of the lightspace $\Ng$ for AdS spacetime  is widely simplified because it lies under the case quoted above. However, this property of cone geodesics may be spoiled for asymptotic AdS spacetimes (because of their behaviour at finite distance). Summing up (Remark \ref{r_asin_AdS}):

\begin{corollary} \label{c1} Any asymptotically  anti-de Sitter spacetime (according to Definition \ref{def:Cuestionespendientes:1}) has a totally lightgeodesic timelike boundary. Thus, it is causally simple and its  lightspace $\Ng$  is Hausdorff.
 Moreover,  $\Ng$ contains an open subset  diffeomorphic to $\R\times T \mathbb{S}^{n-1}$, and this subset is equal to $\Ng$ in the case of AdS spacetime.   
\end{corollary}
The qualitative notions of convexity for the boundary  involving light directions (lightconvexity, eventually strong or total) are conformally invariant, however, their quantitative value depends on the conformal representative, then providing subtle limits of hypersurfaces in $\mathring{M}$ converging to $\partial M$  and boundary geodesics (Remarks \ref{r_totallylightconvex}, \ref{r_AdSlighttotallygeodesic0}, \ref{r_AdSlighttotallygeodesic}). 

\smallskip \noindent Finally, Section \ref{s_last} suggests the inclusion of {\em causally simple Finsler spacetimes with $T_2$-lightspace} as a  step in the standard ladder of causality (straightforwardly extensible to cone structures) between causal simplicity and global hyperbolicity.
An aforementioned example by Hedicke and Suhr \cite[Theorem 2.7]{Hedicke_2019} showed that not all the causally simple spacetimes have a Hausdorff $\Ng$, however, Theorem~\ref{t2} 
provided a large class of  causally simple spacetimes where $\Ng$ is $T_2$. Now, this theorem is complemented with Proposition \ref{p_nocutlight} obtaining:

\begin{corollary}\label{c2}  Causally simple Finsler spacetimes (or cone structures) with $T_2$-lightspace include:
\begin{enumerate}
    \item \label{item_1} The interior $\mathring{M}$ of a globally hyperbolic Finsler spacetime-with-timelike-boundary if and only if the boundary $\partial M$ is lightconvex.
    \item Causally simple Finsler spacetimes whose lightlike geodesics are horismotic (i.e. always maximizing).
\end{enumerate}
\end{corollary}
These results are exemplified  with standard stationary spacetimes (\S \ref{s_stationary}), where Finslerian elements are useful to describe causality (even if the cones come from a Lorentzian metric).  Moreover,  new illustrative counterexamples to  the Hausdorffness of $\Ng$ are provided in \S \ref{s_counterexamples} (see Figures~\ref{fig:2},~\ref{fig:3} and the property quoted in \S \ref{s522}).

\section{Semi-Finsler infinitesimal and local  convexities}\label{s2}

\subsection{Preliminaries on semi-Finsler metrics} 
\label{sec:Introduction}
In this section, we will introduce the basic ingredients for semi-Finsler manifolds that will be required along this paper, following \cite{Javaloyes_2022} (which builds upon \cite{Javaloyes_2019})
as a basic reference. This reference does not consider the case with boundary, so,  we will make a slight adaptation to this case too.

\subsubsection{Conventions on boundaries and  semi-Finsler metrics} 
\label{sec:basic-def}
Let us consider  an $(n+1)$-dimensional connected manifold with boundary $\withboundaryM \equiv \withboundaryM^{n+1}$, 
 let $\partial M$ be its boundary and $\mathring{M}$ its interior.  In general, we will be interested in {\em smooth}  boundaries and elements (as differentiable as necessary), and mention explicitly when only the topological level is used; in particular, $\functionspace(\withboundaryM
)$
denotes the module of smooth functions on $M$. As usual in the setting of spacetimes with boundary (for example,  \cite{RMI_AFS,SolisPHD}), we can assume with no loss of generality that $\withboundaryM$  
lies inside a larger manifold $\entorno{M}$ with the same dimension,  being $\partial M$ the topological boundary of $M$ in $\entorno{M}$; as a rule, elements on $\entorno{M}$ can be regarded as restrictions of elements in $\entorno{M}$ except if otherwise specified.
For any  $p\in M$, $T_p\withboundaryM$ denotes its (n+1)-tangent space (accordingly the tangent bundle is $T\withboundaryM=\cup_{p\in \withboundaryM}T_p \withboundaryM$), and the subspace of vectors tangent to the boundary is denoted  $T_{p} \partial M\subset T_{p} \withboundaryM$, any  $w\in T_{p} \withboundaryM\setminus T_{p} \partial M$ is called \emph{transverse} to the boundary. For $p\in \partial M$, an
{\em adapted chart $(U,\varphi=(x^1,\dots,x^{n+1}))$} (centered at $p$) satisfies:

\begin{enumerate}[label=(\roman*)]
\item $\varphi:U \rightarrow 
\left\{ (x^{1},\dots,x^{n+1})\in \mathbb{R}^{n+1}:x^{n+1}\geq 0 \right\}$,
\item $q\in \partial U:= U\cap(\partial M)$ if and only if $x^{n+1}(q)=0$.
\end{enumerate}
Any adapted chart $(U,\varphi=(x^1,\dots,x^{n+1}))$ also defines a chart on $TU\subset T\withboundaryM$, namely, $(x,y)=(x^1,\dots,x^{n+1},y^1,\dots,y^{n+1})$, where $x^{n+1}\geq 0$ but no restriction on $y^{n+1}$ (or any $y^i$) appears.

Next, we introduce the definition of a semi-Finsler metric on $\withboundaryM$, consistently with\footnote{Here, however, we use the term {\em semi-Finsler}, instead of {\em pseudo-Finsler}, used in \cite{Javaloyes_2022} and its previous bibliography. Our present choice is consistent with the name {\em semi-Riemannian}, introduced by O'Neill \cite{oneill} for the  extension of Riemannian Geometry to (constant non-degenerate) signature. For us, the name pseudo-Finsler, as well as pseudo-Riemannian, remains for  more general objects, such as  signature changing metrics, which might  degenerate somewhere).} \cite{Javaloyes_2022}.

\begin{definition}
\label{def:main2:5}
Let $\withboundaryM^{n+1}$ be a manifold with  boundary and let  $\pi:T\withboundaryM\rightarrow \withboundaryM$ the natural projection. 
Let ${\withboundaryA}\subset T\withboundaryM\setminus \left\{ \mathbf{0}  \right\}$ be an open subset which is conic (namely, for every $v\in {\withboundaryA}$ and $\lambda>0$, $\lambda v\in {\withboundaryA}$) and satisfies that $\pi({\withboundaryA})=\withboundaryM$. We will say that a smooth function\footnote{ The use of $\bar A$ or $A$ here will be irrelevant, as the smoothness al $\partial A$ of $g_v$ in \eqref{eq:21} will imply its extendibility to an open neighborhood maintaining its non-degeneracy. However, the distinction $A, \overline{A}$ is maintanined here for consistency with the Lorentz-Finsler case, where this convention stresses the irrelevance of non-causal  directions.} $L:\cls{A}\rightarrow \mathbb{R}$ defined on the closure $\cls{\withboundaryA}$ of $A$ in $T\withboundaryM\setminus \left\{ 0 \right\}$ is a \emph{semi-Finsler metric} if:

\begin{enumerate}[label=(\roman*)]
\item $L$ is positive homogeneus of degree $2$: that is, $L(\lambda v)=\lambda^2L(v)$ for every $v\in \cls{\withboundaryA}$ and $\lambda>0$.
\item for every $v\in \cls{\withboundaryA}$, the fundamental tensor of $L$, defined as
\begin{equation}
	\label{eq:21}
	g_v(u,w)=\frac{1}{2}\dfrac{\partial^{2}}{\partial t\partial s}L(v+tu+sw)|_{t=s=0}
\end{equation} with $u,w\in T_{\pi(v)}\withboundaryM$, is non-degenerate with constant signature $(+,\dots, +, -, \dots , -)$.
\end{enumerate}
The pair $(M,L)$, emphasized notationally as a triple $(\withboundaryM,L,{\withboundaryA})$,  will be called a \emph{semi-Finsler manifold (with boundary)}.
\end{definition}
\begin{remark}
\label{rem:main2:1} {\em 
(1) The minimum  differentiability  is  $C^{3,1}$ for $L$ and $C^{2,1}$ for regular functions whose levels define the boundaries, so that $g$ becomes $C^{1,1}$ (and can be imposed on any extension without boundary $\entorno{M}$).  

(2)  The signature of $g$ in this definition is imposed to be constant, as $\withboundaryA$ is not assumed to be connected. Anyway,  this restiction will not be necessary,  because we will work either locally (in this section) or with Lorentz signature $(+,-,\dots -)$ (in the others). 

(3) $L$ can be extended continuously as equal to zero at the zero-section, however, it would not be differentiable there except if  $L$ comes from a semi-Riemannian metric (see \cite[Remark 2.13]{JS20}).

(4) For any coordinates $(x,y)$  in $TM$ as above, $g$ can be written as
$$
g_{(x,y)}= g_{ij}(x,y) dx^i dx^j 
$$
where $g_{ij}(x,y)$ is positive 0-homogeneous in $y$ (that is, $g_{ij}(x,\lambda y)= g_{ij}(x,y)$ for $\lambda>0$). Thus, $g_{ij}(x,y)$ can be regarded as the coordinates of a (2-covariant symmetric non-degenerate) {\em anisotropic}  tensor, i.e., a tensor in each direction $y \in \bar A_x$.
} 
\end{remark}

\subsubsection{Anisotropic lightlike structures in arbitrary signature}

By using the semi-Finsler metric as in the semi-Riemannian case, we can define the causal character of a tangent vector $w$ depending on the sign of $L(w)$. As we will be interested  later in Lorentzian signature $(+,-,\dots -)$, $w$ is defined as timelike (resp.  lighlike; causal; spacelike; null) when $L(w)>0$ (resp. $L(w)=0; L(w)\geq 0$; $L(w)<0$; $w$ is either causal or 0), all of them but the 0 vector living in the conic region $A$, consistently with \cite{JS20}. When working in arbitrary signature, the names timelike and spacelike only distinguish the sign of $L(w)$. However, in Lorentzian signature 
 they permit to define causality (see \S \ref{sec:conestructure}) extending relativistic notions. Anyway, the sets of  lightlike vectors   
\begin{equation}
\label{eq:22}
\lc{p}:=\left\{ v\in \cls{\withboundaryA}_p:L(v)=0 \right\}, \qquad
\lc{}   	:= \cup_{p\in \withboundaryM} \lc{p}
\end{equation}
named,  respectively, the cone (at $p$) and the cone structure (of $\withboundaryM$), 
will be of special interest here, for arbitrary signature. Along this section, we will be interested in  properties of $\lc{}$ which are local in the tangent bundle. However, the following considerations are interesting for the global structure of $\lc{p}$ at the tangent space of each point $p$.
Let start recalling the semi-Riemanian case. 

\begin{remark} \label{r_sign_kl}
{\em For a semi-Riemmannian metric $g$ of signature $(k+1,l+1)$ with $k + l = n - 1, k,l \geq 0$, $\lc{p}$ yields a conic hypersurface of $T_pM$ whose positive projectivization (into a sphere bundle) is  homeomorphic to the product of spheres $S^k\times S^{l}$. In the Lorentzian case $l=0$, and $S^{l=0}$ has two points, each one corresponding to a standard Lorentzian cone, which can be locally labelled as a future or past cone. When the Lorentzian manifold is time orientable, $\lc{}$ has two connected components,  each one corresponding to a cone structure in our sense.  }
\end{remark}

Going from the semi-Riemannian to the semi-Finsler case, a general notion of {\em causal structure} $(M, \mathbf{C})$ for any signature, as defined  in \cite{Makh}, turns out essentially equivalent to ours. Namely,  given a
$(n+1)$-manifold M, with $n+1 \geq 3$, a {\em causal structure of signature $(k+1; l+1)$}
with $k, l \geq 0$ and $k + l = n - 1$,  is a sub-bundle $\mathbf{C}$ of the (positively) projectivized tangent bundle $PTM$ whose fibers $\mathbf{C}_p$ ($p\in M$) are projective hypersurfaces and have projective  second fundamental form of signature $(k, l)$ (thus, the latter is non-degenerate).

Projectivizing each $\lc{p}$ in \eqref{eq:22}, it is straightforward to check that our definition of lightlike cone provides a causal structure in this sense. Conversely, such a causal structure  $\mathbf{C}$ can be recovered from a semi-Finsler metric. Indeed,  $\mathbf{C}$ can be regarded as a submanifold in the sphere bundle for an auxiliary Riemanian metric. Then, the required semi-Finsler metric $L$  will be the signed distance $d$ to $\mathbf{C}$ (in a neighborhood of $\mathbf{C}$ where $d$ remains smooth), extended as a 2-homogenous function to a conic subset of $TM$; obviously, such an $L$ is highly non-unique.

In particular, the case $(p=n-1, q=0)$ (thus, the projectivized second fundamental form is a positive definite) corresponds to the projectivization of the cone structures of Lorentz-Finsler metrics to be considered in Section \ref{sec:lorentzfinsler}.
When $k,  l> 0$ the  classical methods of causality (as in the mentioned section) do not apply to $(M,\mathbf{C})$, because its corresponding cone structure would be {\em totally vicious}, i.e., closed timelike curves through any point would exist. However,  the causal structure $(M,\mathbf{C})$ can still be studied by using Cartan methods, as done in \cite{Makh}.


\subsubsection{Geodesic spray and anisotropic connections
} \label{sec:geod spray anis}

As in the semi-Riemannian setting, the semi-Finsler metric $L$ admits (parametrized)  geodesics, defined as the curves $\gamma$ which are  critical for the energy functional $\int L(\dot \gamma)$, and, thus, are characterized by  its Euler-Lagrange equation, i.e. the geodesic equation \eqref{eq:65} below. Then, $L(\dot \gamma)$ must be a constant, so that one can speak on {\em timelike, lightlike, causal} or {\em spacelike geodesics} accordingly.  
Strictly speaking, we will require only geodesics and a related notion of Hessian for a function which are explained below. However, we will work with the Hessian as an anisotropic tensor field obtained from an anisotropic connection $\nablaM$, which is explained further in the Appendix  \ref{A_anisotropic}.

Let us introduce the geodesic spray associated with $L$, following \cite[ Section 6.1]{Javaloyes_2022} and using the convention of sum in repeated indeces.
\begin{definition}
\label{def:main2:13}
Let $(\withboundaryM,L,{\withboundaryA})$ be a semi-Finsler manifold with fundamental tensor $g$. Its \emph{geodesic spray} is the vector field $G$ over $\overline{\withboundaryA}$ (i.e., a section $\overline{\withboundaryA}\rightarrow T\overline{\withboundaryA}$) satisfying:

\begin{enumerate}[label=(\roman*)]
\item ${G}$ is a second order equation, that is, ${G}$ can be written in $\withboundaryA\cap TU$ as:
\begin{equation} \label{eq:62}
	{G}_{(x,y)}=y^{k} \frac{\partial}{\partial x^{k}}|_{(x,y)} - 2 {G}^{k}(x,y)\frac{\partial}{\partial y^{k}}|_{(x,y)}.
\end{equation}
\item ${G}^{k}(x,y)=\Gamma_{ij}^{k}(x,y)y^{i}y^{j}$, where $\Gamma_{ij}^{k}$ are the so-called \emph{formal Christoffel symbols} which are $0$-homogeneous functions on the variable $y$ (i.e., $\Gamma_{ij}^{k}(x,\lambda y)=\Gamma_{ij}^{k}(x,y)$ for $\lambda>0$) given by:
\begin{equation}
	\label{eq:64}
	\Gamma_{ij}^{k}(x,y) = \frac{1}{2}g^{kl}|_{(x,y)} \left( \frac{\partial g_{li}}{\partial x^{j}} + \frac{\partial g_{lj}}{\partial x^{i}} - \frac{\partial g_{ij}}{\partial x^{l}} \right)|_{(x,y)}.
\end{equation}
\end{enumerate}
In particular, $G$ is a {\em spray}, namely, $G$ satisfies \eqref{eq:62}  with each function  $G^k$ positively homogeneous of degree two in $y$ (that is,  $G^k(x,\lambda y)=\lambda^2 G^k(x,y)$).
\end{definition}
Previous geodesic spray allow us to recover geodesics in the following way: a curve $\gamma$ is a geodesic for the geodesic spray ${G}$ if its velocity, seen as a natural lift to $TT\withboundaryM$ is an integral curve of ${G}$. In coordinates, $\gamma(t)=(\gamma^1(t),\dots,\gamma^{n+1}(t))$ is a geodesic if and only if it satisfies the equations:
\begin{equation}
\label{eq:65}
\ddot{\gamma}^{k} + \Gamma_{ij}^{k}(\gamma,\dot{\gamma})\dot{\gamma}^{i}\dot{\gamma}^{j}=0, \qquad \hbox{for $k=1, \dots , n+1$,}
\end{equation}
which are formally equal to the geodesic equations for a semi-Riemannian metric, taking into account that now, the symbols $\Gamma_{ij}^{k}(\gamma,\dot{\gamma})$ depend also on the (direction of the) velocity, as they are computed in \eqref{eq:64} from the direction-dependent fundamental tensor $g$.

\begin{remark} \label{rem:conformal_lightlikepregeod} {\em Semi-Finsler geodesics not only present a well defined type associated with the constancy of  $L(\dot \gamma)$ (as commented above) but also other relevant semi-Riemannian analogies for lightlike {\em pregeodesics} (i.e., geodesics up to reparametrization).    Lightlike pregeodesics for a semi-Riemannian $g$ are invariant under any conformal change $g \rightarrow g^*:=e^{2u}g$  where $u\in \functionspace(\withboundaryM)$, and its conjugate points and multiplicity remains invariant too \cite[Theorem 2.6]{MS}; these results are valid in any signature, even though Lorentzian signature  interpretations on causality permit some simplifications. All these properties can be extended to two semi-Finsler metrics $L,L^*$ which are {\em anisotropically conformal}, i.e.,  $L^*:=e^{2u}L$ where now $u$ is a (necessarily 0-homogeneous) smooth function on $A$, see \cite{Javaloyes_2020},  \cite[Sect. 3.3 and 6.2]{JS20}.
}
\end{remark}

\begin{definition}\label{d_hess} (1) A vector field $V\in \mathfrak{X}(M)$ is {\em $A$-admissible} if $V_p\in A$ for all $p\in M$

(2) The (Berwald)  Hessian associated with $L$ of a smooth function $\phi \in  \functionspace(\withboundaryM)$
 is the symmetric 2-covariant anisotropic tensor field\footnote{{\em Anisotropic tensor} means that at each (oriented) direction provided by the A-admissible vector field $V$ one has an ordinary tensor. Notice that if $V,\bar V$ are $A$-admissible and $V_p=\bar V_p$  then $\hbox{{\rm Hess}}^V|_p=\hbox{{\rm Hess}}^{\bar V}|_p$  and, if $\lambda>0$, $\hbox{{\rm Hess}}^V=\hbox{{\rm Hess}}^{\lambda V}$. Analogous considerations hold for anisotropic connections (see \eqref{e_chris}, \eqref{e_anisV}).}
$\hbox{{\rm Hess}}$ defined  in coordinates by
\begin{equation}
\label{eq:71}
	\hbox{{\rm Hess}}^V_\phi(Z,W)=
	Z^{i} W^{j}\partial_{i}\partial_{j}\phi -\Gamma_{ij}^{k}(V) Z^{i}W^{j} \partial_{k} \phi.
\end{equation}
for all  $Z,W,V\in \mathfrak{X}(\withboundaryM)$ with $V$ being $A$-admissible and $\Gamma_{ij}^{k}(V)$  computed from \eqref{eq:64} with $V=(x,y)$. 
\end{definition}

\begin{remark}\label{r_nulconvex_depend_geodeic_vf}
{\em (1) It is easy to check the independence of coordinates in \eqref{eq:71}. More intrinsically, the formal Christoffel symbols \eqref{eq:64} provide a (Berwald) affine connection $\nabla^V$ for each A-admissble $V$ (essentially, this is the Berwald  anisotropic connection, see Appendix \ref{A_anisotropic}) and, then, $\hbox{{\rm Hess}}^V$ can be written as
\begin{equation}
\label{eq:71bis}
	\hbox{{\rm Hess}}^V_\phi(Z,W)=(\nablaM^V d\phi)(Z,W)=(\nablaM^V_Z d\phi)(W)-d\phi(\nablaM^V_ZW)=Z(W(\phi))-(\nablaM^V_Z W)(\phi).
\end{equation}
(2) When $\phi$ is a coordinate function, say $\phi=x^{n+1}$, $\hbox{{\rm Hess}}^V$ becomes,  for $V=Z=W$ $A$-admissible:   \begin{equation}
\label{eq:71bis2}\hbox{{\rm Hess}}^V_{x^{n+1}}(V,V)=-\Gamma_{ij}^{n+1}(V)V^iV^j.\end{equation} Considering $\gamma(t)=(\gamma^1(t),\dots,\gamma^{n+1}(t))$ a geodesic with initial velocity $w=V_p$ and recalling \eqref{eq:65}, previous Hessian agrees at $p$ with $\ddot{\gamma}^{n+1}(0)$ and depends only on $w$. Clearly,  this expression, 
to be used later, 
relies only on the associated geodesic 
spray\footnote{Thus, \eqref{eq:71bis2} is equal for both the Berwald and the Levi-Civita-Chern anisotropic connections (see the Appendix~\ref{A_anisotropic})
as   both share $G$.}. 
}
\end{remark}

\subsection{Infinitesimal convexities for the boundary}
\label{sec:infinitesimal}

Next, we will introduce the essential notions of 
infinitesimal light, time and space convexity.  As in the case of  semi-Riemannian manifolds, first  we require a notion of second fundamental form for $\partial M$, which determines how this hypersurface bends. This form will depend on the choice of a rigging vector field, but convexity will be independent of that choice.

\subsubsection{Second fundamental form depending on a rigging vector field $\eta$}
First, let us see that any anisotropic connection (say the Berwald connection, see Remark \ref{r_nulconvex_depend_geodeic_vf}) 
induces one  on its  boundary $\partial M$ after fixing a vector field $\eta$  transverse to $\partial M$. This is a local notion, thus, we will focus on a point $p\in \partial M$ and    an adapted chart $(U,\varphi)$ centered at $p$ as in \S \ref{sec:basic-def} (extended beyond $\partial M$ to $\widehat{M}$).

\begin{definition}
A \emph{rigging vector field} $\eta$ in $U$ is a vector field $\eta:U\rightarrow TU \subset T\withboundaryM$ so that $\eta_q$ is transverse to the boundary and directed inwards for all $q\in \partial U$.
\end{definition}

Let $X,Y \in \mathfrak{X}(\partial U)$ be   two vector fields on the boundary, an extend them to   $\widehat{X},\widehat{Y}\in \mathfrak{X}(U)$ on  $U$. For any $\withboundaryA$-admissible vector field $V\in \mathfrak{X}(U)$, the following  decomposition holds on $\partial U$:
\begin{equation}
\label{eq:1}
\nablaM^{V}_{\widehat{X}}\widehat{Y} = (\nablaM^{V}_{\widehat{X}} \widehat{Y})^{T} + \second{\eta}^{V}(\widehat{X},\widehat{Y})\eta,
\end{equation}
 where $(\nablaM^{V}_{\widehat{X}} \widehat{Y})^{T}\in T\, \partial M$, i.e., we are considering the  decomposition  given by the rigging vector field and the boundary $T \partial M$ of the vector field $\nablaM^V_{\widehat{X}} \widehat{Y}|_{\partial U}$. Then the {\em induced connection} $\nablai{\eta}$ and {\em second fundamental form} $\second{\eta}$  associated with the rigging vector $\eta$ are defined, resp., as:

\begin{equation}
\label{eq:1bis}
\nablai{\eta}^{V}_XY := (\nablaM^{V}_{\widehat{X}}\widehat{Y})^{T}|_{\partial U}, \qquad \qquad \second{\eta}^{V}(X,Y)=\second{\eta}^V(\widehat{X},\widehat{Y})|_{\partial U}. \end{equation}
 Several observations are in order. 
 (a) As standard, these definitions are independent on the chosen extensions  $\widehat{X}$ and $\widehat{Y}$ around each point $p$ and, moreover, the expressions at each $p\in \partial M$ depend on $V_p$ rather  than its extension $V$ (consistently with  \eqref{e_anisV}). 
  (b) Rigging vectors for the hypersurface $\partial M$  can be chosen {\em normal} in  standard Finsler Geometry and, when $\partial M$ is non-degenerate, in  semi-Riemannian Geometry. However, in the semi-Finsler case these directions may not exist and, indeed, the rigging vector can be chosen outside $\overline{A}$.
  (c) 
 Thus,  our  induced anisotropic connection may differ from  the intrinsic anisotropic one which would inherit $\partial M$ in case that the restriction of $L$ is non-degenerate therein.

\begin{proposition}
\label{prop:main2:5}
Let $(\withboundaryM,L,\withboundaryA)$ be a semi-Finsler manifold with a smooth boundary $\partial M$, choose $p\in \partial M$ and consider two rigging vector fields $\eta,\overline{\eta}$, defined in a neighborhood $U$ of $p$. Then, for any  $\withboundaryA$-admissible vector field $V$:
\begin{equation}
\label{eq:3}
\left\{ 
  \begin{array}{l}
    \nablai{\eta}^{V}_{X} Y =  \nablai{\overline{\eta}}^{V}_{X} Y + \second{\overline{\eta}}^{V}(X,Y) \overline{\eta}^{T_\eta}\\
\second{\eta}^{V}(X,Y)= B \second{\overline{\eta}}^{V}(X,Y),
  \end{array}
\right.\quad \hbox{for all $X,Y\in \mathfrak{X}(\partial M)$.}
\end{equation}
where $\overline{\eta}$ is uniquely decomposed as $\overline{\eta}=\overline{\eta}^{T_\eta}+B\eta$, with the vector field $\overline{\eta}^{T_\eta}$  on $\partial U$ tangent to $\partial U \subset \partial M$ and $B$  a  function on $\partial U$ (satisfying $B\geq 0$, as both riggings point inwards). 
\end{proposition}
\begin{proof} Just compare \eqref{eq:1}, \eqref{eq:1bis} with
\begin{equation}
	\label{eq:2}
	\begin{split}
		\nablaM^{V}_X Y =& \nablai{\overline{\eta}}^{V}_{X} Y + \second{\overline{\eta}}^{V}(X,Y) \overline{\eta} = \nablai{\overline{\eta}}^{V}_{X} Y + \second{\overline{\eta}}^{V}(X,Y)\left(\overline{\eta}^{T_{\eta}}+B \eta  \right)\\ =&  \left( \nablai{\overline{\eta}}^{V}_{X} Y + \second{\overline{\eta}}^{V}(X,Y) \overline{\eta}^{T_{\eta}}\right) +B\; \second{\overline{\eta}}^{V}(X,Y) \eta .
	\end{split}
\end{equation} 

\end{proof}

\subsubsection{Definition and characterization of infinitesimal convexities}

From Proposition \ref{prop:main2:5}, even if the notion of the second fundamental form depends on the rigging vector field $\eta$, the \emph{sign} of such a fundamental form remains invariant under changes of $\eta$. We can then introduce the following notion of convexity, which have no dependence on the choice for the rigging vector field (so, $\eta$ will be removed from the notation once there is no ambiguity).
\begin{definition}
\label{def:main2:12}
Let $(\withboundaryM,L,{\withboundaryA})$ be a semi-Finsler manifold with smooth boundary $\partial M$. A boundary point $p\in \partial M$ is {\em infinitesimal  lightconvex (resp. timeconvex,  spaceconvex) at $p$} if for any $A$-admissible lightlike (resp. timelike, spacelike) vector $w\in T_p\partial M\cap \bar A_p$, we have $II^{w}(w,w)\geq 0$,  where $II^w$ denotes  the second fundamental form with respect to one, and then all, (inwards-directed) rigging vector field $\eta$. 
If this property holds at each $p\in \partial M$, then $\partial M$ is {\em infinitesimal lightconvex (resp. timeconvex, spaceconvex)}. 
\end{definition}

Infinitesimal convexity can be easily characterized in terms of the Hessian of the coordinate defining the boundary in any adapted chart (recall Definition \ref{d_hess} and Remark \ref{r_nulconvex_depend_geodeic_vf}).

\begin{proposition}
$\partial M$ is infinitesimally lightconvex (resp. timeconvex, spaceconvex) at $p$ if and only if for one (and then any) adapted chart $(U,\varphi=(x^1,\dots,x^{n+1}))$ centered at $p$,
$$Hess^{w}_{x^{n+1}}(w,w)\leq 0 \quad \hbox{for any lightlike (resp. timelike, spacelike)} \; w\in T_p\partial M. $$ 
\end{proposition}
\begin{proof}
For any $A$-admissible  $V$ and any vector field $W$ 
on $U$  tangent to $\partial M$ in $\partial U$, using \eqref{eq:71bis},
\begin{equation}
\label{eq:4}
\begin{array}{rl}
\hbox{{\rm   Hess}}^{V}_{x^{n+1}}(W,W)
^V_{Z}(dx^{n+1})W 
=& W(W(x^{{n+1}}))-dx^{n+1}({\nablaM}^{V}_W W)\\ =& -dx^{n+1}(\nablai{\eta}^{V}_W W + \second{\eta}^V(W,W)\eta)\\
  =&-\second{\eta}^V(W,W)dx^{n+1}(\eta).
\end{array}
\end{equation}
Recall $dx^{n+1}(\eta)>0$ (as $\eta$ is inner-directed) and choose $V=W$,  
an $A$-admissible vector field extending $w$.
\end{proof}


The invariance of null convexity under conformal changes will be a key property for causality, as this means that it depends only on the lightlike cone structure. 
\begin{proposition}\label{p_invar_conf_lightconvex}
	Infinitesimal lightconvexity  and strong\footnote{The name  {\em strict} (local) convexity is commonly referred to the case when, up to the initial point,  $\gamma$ must lie in $\widehat{M}\setminus M$. For this reason, {\em strong} convexity refers here to the case when $II$ is positive definite.} infinitesimal lightconvexity  are invariant under anisotropical conformal changes of the metric.
\end{proposition}
\begin{proof}
From Remark \ref{r_nulconvex_depend_geodeic_vf} and formula \eqref{eq:71bis2}, Hess$^{w}_{x^{n+1}}(w,w)$ can be regarded as $\ddot{\gamma}^{n+1}(0)$ for a lightlike geodesic $\gamma$. As the sign of $\ddot{\gamma}^{n+1}(0)$ is invariant for reparameterizations of $\gamma$, the result follows from the invariance of lightlike pregeodesics under anisotropic conformal changes (Remark \ref{rem:conformal_lightlikepregeod}).
\end{proof}

\subsection{Equivalence  of infinitesimal and local convexities}

In the Riemannian or Finslerian setting, a commonly used notion of convexity for the hypersurface $\partial M$ in $\widehat M$ is {\em (local) convexity at $p$}, which means that no geodesic $\gamma$ tangent to $\partial M$ at $p$ enters the interior $\mathring{M}$ in a neighborhood of $p$ (see, e.g., \cite{BGS}). When this property holds for every $p \in \partial M$, we say that $\partial M$ is {\em locally convex}. This can be extended directly to the  semi-Finslerian setting.


\begin{definition}
Let $(M,L,A)$ be a semi-Finsler manifold with boundary.
$\partial M$ is {\em (locally) lightconvex (resp. timeconvex, spaceconvex)} if  any {\em lightlike} (resp.  {\em timelike, spacelike}) geodesic $\gamma$ in $M$ with initial velocity tangent to $\partial M$ 
must be contained\footnote{\label{f_localconvexity} In principle,  this property should  be imposed only for some neighborhood around the initial point $p\in \partial M$. However, the global statement follows applying this property to the first  point where $\gamma$  would leave $\partial M$.} in $\partial M$.
\end{definition} 
\begin{proposition}\label{prop_strictconv} If a vector $w\in A_p \cap T_pM$ satisfies $II^w(w,w)<0$ then the geodesic $\gamma$ with initial velocity $w$ is not contained in $\partial M$, moreover, at least locally, the only intersection between $\gamma$ and $\partial M$ is the point $p$.

Thus, (for geodesics of each  type) local convexity implies infinitesimal convexity .
\end{proposition}
\begin{proof} Let $({U},\varphi=(x^1,\dots,x^{n+1}))$ be an adapted chart and $W$ be a vector extending locally in $\widehat{M}$ the velocity of $\gamma$. By \eqref{eq:4}, the function $\gamma^{n+1}(t)=x^{n+1}(\gamma(t))$ satisfies 
  \begin{equation}
\label{eq:10}
          \gamma^{n+1}(0)=0, \quad
           \dot{\gamma}^{n+1} (0) = 0, \quad 
          \ddot{\gamma}^{n+1}(0)=Hess^{\dot{\gamma}}_{x^{n+1}}(\dot{\gamma}(0),\dot{\gamma}(0))>0.
\end{equation}
Then $\gamma^{n+1}(t)>0$ for all $t\in (-\epsilon,\epsilon)$ but $0$, and all the conclusions follow.
\end{proof}
The converse would be trivial from the previous proof under  strong convexity $\second{}^w(w,w)>0$. However, the proof in \cite{Bartolo_2010} holds for the standard Finsler case, yielding even the following slightly more general result.  

\begin{theorem}
\label{prop:main2:6}
Let $(\withboundaryM,L,A)$ be a semi-Finslerian manifold with  boundary, and consider a geodesic $\gamma:[0,b]\rightarrow \withboundaryM$  with $\gamma(0)\in \partial M$ and $\dot{\gamma}(0)$ tangent to $\partial M$. Assume that there exists an open subset $\mathcal{U}$ of $ A \cap T(\partial M)$ with $\dot{\gamma}(0)\in \mathcal{U}$ and satisfying that, for any $w\in \mathcal{U}$,  $II^w(w,w)\geq 0$. Then, $\gamma|_{[0,\epsilon)}\subset \partial M$ for some $\epsilon>0$.
\end{theorem}
\begin{proof}
    Notice that the mentioned proof in \cite{Bartolo_2010} used only: (a) a local reasoning in $TM$ around the initial velocity of the geodesic (and convexity for the vectors therein therein), and (b)  a projection in a direction transverse to the boundary (which may be unrelated to the Finsler metric). So, it works directly in our case.
\end{proof}

This  applies to the infinitesimal convexity of the following type of vectors in $T \partial M$ (as they yield an   open  subset) thus implying the equivalence for the corresponding type of local convexity: (a)  all the vectors, i.e., $A\cap T\, \partial M$ (b) timelike vectors, (c) spacelike vectors. 

For (non-strong) infinitesimal  lightconvexity, the  equivalence  was achieved  by Hintz and Uhlmann \cite{Hintz_2017} 
for  Lorentz manifolds with timelike boundary. The proof can be extended  to the (indefinite) semi-Finsler case using the following notion. 

\begin{definition}\label{d_no_degenerate}
    An indefinite semi-Finsler manifold $(\withboundaryM,L,{\withboundaryA})$ has 
    {\em non-degenerate boundary} if the lightcone 
    $\lc{p}$ intersects transversely 
    $T_p \partial M$, that is, if the kernel of $g_w(w, \cdot)$ is not equal to $T_p \partial M$, at any lightlike direction $w\in T_p \partial M$.
\end{definition}


\begin{theorem}
\label{prop:main2:4}
Let $(\withboundaryM,L,{\withboundaryA})$ be a semi-Finsler manifold with indefinite non-degenerate smooth boundary. If $\partial M$ is infinitesimally lightconvex, then any lightlike geodesic $\gamma:[0,b]\rightarrow \withboundaryM$ with $\gamma(0)=p\in \partial M $ and $\dot{\gamma}(0)\in T_p\partial M$ is contained in $\partial M$.
\end{theorem}

\begin{proof} Following  \cite[Lemma 2.3 (2)]{Hintz_2017}, first we will prove that a  geodesic $\gamma_0$ starting at $p\in \partial M$  pointing outwards $\withboundaryM$ with  velocity close to $\gamma$ cannot return to  $\partial M$ (at least in a certain fixed neighborhood  of   $p$ in an extension $\widehat{M}$ of $M$). Then, a   limit argument will complete the proof.

For the  first part, let $(U_0,\varphi=(x^1,\dots,x^{n+1}))$ be an adapted chart around $p$ in $\widehat{M}$ and $U\subset U_0$  a normal neighborhood of $p$ with compact closure included in  $U_0$. Then, we can regard $U, U_0$ as neighborhoods of $0$ in  $\mathbb{R}^{n+1}$ and use the Euclidean metric $\langle \cdot,\cdot \rangle$ 
and distance.

Let us define a family of hypersurfaces $D_{\epsilon}$ 
that foliate a compact neighborhood 
 $\mathcal{B}$ of $p$ in the region  $ \left\{ x^{n+1}\leq 0 \right\}$ of $U$, each $D_\epsilon$
  strongly   infinitesimally lightconvex therein. Namely, let    
\begin{equation}
\label{eq:11} \begin{array}{cll}
x^{n+1}_{\epsilon}:= x^{n+1} + \epsilon \left( 1- \delta^{-2} |x|^2 \right)
     & & |x|:=\sqrt{\Sigma_{i=1}^n(x^i)^2}<\delta^2 \\
 D_{\epsilon}:=(x_{\epsilon}^{n+1})^{-1}(0) &   & 0\leq \epsilon\leq \delta \\
 \mathcal{B}=\left\{ -\delta/2 \leq x^{n+1} \leq 0, |x| \leq \delta^2 \right\} & &
\end{array} \end{equation}
for some small  $\delta>0$ to be determined.
$D_{\epsilon}$ is  $\epsilon$-close to $D_0=\partial M \cap \{|x|<\delta^2\}$, as $dx^{n+1}_{\epsilon} = dx - 2\epsilon \delta^{-2} \sum_{i=1}^n x^i dx^i$ and $|x^i|/\delta^2<1$. 
Let $\lc{}^u$ be the set of all the $\langle \cdot,\cdot \rangle$-unit lightlike vectors in $U$. We can find  
an open neighborhood $\mathcal{O}^u$ of  $\lc{}^u \cap  T(\partial M)  $ in $\lc{}^u$ such that, for all $w\in \mathcal{O}^u$, 
\begin{equation}
  \label{pend:eq:10:nuevasec}
  \mathrm{Hess}_{x^{n+1}}^w(w,w) \leq C_1 \epsilon,
\end{equation}
(because the inequality to $0$ follows in $T\, \partial M$ by lightconvexity), $\sum_{i=1}^n (w^i)^{2}\geq C_2$ for some constant $C_2>0$ (as $w$ satisfies $\sum_{i=1}^{n} (w^i)^2=1$ on $\lc{}^u$) 
 and,  for every $q \in \pi(\mathcal{O}^u)$, the set $\lc{q}^u \cap T D_\epsilon$ is contained in $\mathcal{O}^u$ 
(as $D_{\epsilon}$ is $\epsilon$-close to $D_0$).  Choosing $\delta>0$ small, we have 
$\mathcal{B}  \subset \pi(\mathcal{O}^u)$ and, for $\epsilon>0$,

\begin{equation}
\label{eq:13}
\begin{split}
  \mathrm{Hess}_{x^{n+1}_{\epsilon}}^w(w,w)=&\mathrm{Hess}_{x^{n+1}}^w(w,w) -2\epsilon\delta^{-2} \left(\sum_{i=1}^ndx^i(w)^2+x^i \mathrm{Hess}_{x^i}^w(w,w)\right)\\
  \leq & \epsilon \left(C_1-C_2\delta^{-2}  \right)<0,
\end{split}
\end{equation}
(a smaller $\delta$ can be chosen  to ensure that  the first parentheses is positive). This holds for any $w$ lightlike   $\langle \cdot,\cdot \rangle$-unit and tangent to $D_\epsilon$ in $\mathcal{B}$, thus proving strong lightconvexity for each $D_\epsilon$.

 Now, consider any $\langle \cdot,\cdot \rangle$-unit reparametrized lightlike geodesic $t\mapsto \gamma_0(t)$ which starts at $p$ and points outwards $\withboundaryM$. Let us show that $\gamma_0$ 
 cannot return to $\partial M$ in $\mathcal{B}$. Otherwise, consider the value of $\epsilon$ on $\gamma_0$ as  a function  in the parameter $t\in [0,b)$, that is, putting  $x^{n+1}_{\epsilon}=0$ in \eqref{eq:11},
$$
\epsilon(t):= \frac{  -  x^{n+1}(\gamma_0(t))}{1-\delta^{-2}|\gamma_0(t)|_0^2}. \qquad t\in [0,b).
$$
Before $\gamma$ returns to $\partial M$, $\epsilon$  must reach a maximum $\epsilon_*=\epsilon (t_*)$, and  $ \dot \gamma_0(t_*)$ must be tangent to $D_{\epsilon_*}$ therein. This  contradicts the strong lightconvexity of $D_{\epsilon^*}$, because it would force  $\gamma_0$ to move away in the direction of increasing $\epsilon$, as required.  Notice also that,  as $\gamma_0$ must abandon $\mathcal{B}$, its (maximal) domain  satisfies $b> b_0:=$ Min$\{\delta/2,\delta^2\}$.

Once this first step has been completed, recall that $\gamma$ starts with a lightlike velocity tangent to $\partial M$.  By the condition of non-degeneracy of $\partial M$, $\gamma$ can be approximated by lightlike geodesics $\gamma_k$ starting at $p$, that is, $\dot\gamma_k(0)\rightarrow \dot\gamma(0)$, with velocities $\dot\gamma_k(0)$ pointing outwards $\withboundaryM$.
Taking them also  $\langle \cdot,\cdot \rangle$-unit parametrized in $[0,b_0]$, we have proved that all $\gamma_k((0,b_0))$  are  contained in $\widehat{M}\setminus \withboundaryM$, thus, $\gamma([0,b])\subset\widehat{M}\setminus \mathring{M}$ and the result follows (recall footnote \ref{f_localconvexity} too). 
\end{proof}
The proof  of lightconvexity in Th. \ref{prop:main2:4} can be directly extended to time and space convexities (even though it is unclear either  if it could be extended to prove Th. \ref{prop:main2:6} or viceversa.
Anyway, Theorem \ref{t1} follows from these two theorems and Proposition \ref{prop_strictconv}.


\section{Cone structures and 
causal simplicity}\label{s3}

\subsection{Preliminaries on Finsler spacetimes with boundary}
\label{sec:lorentzfinsler}
Next, we will introduce the basics of cones, Lorentz-Finsler metrics and causality following  \cite{JS20} and adding the pertinent observations to include the boundary $\partial M$.

\subsubsection{Cone structures and Lorentz-Finsler metrics}
\label{sec:conestructure}
 We follow \cite[Def 2.1 and 2.7]{JS20} (some small redundancies in our definition are discussed therein).

\begin{definition}
\label{def:main2:4}
A cone structure $\lc{}$ defined over a manifold with boundary ${M}$ is an embedded hypersurface of $T{M}$ satisfying that, for each $p\in {M}$ and $v\in \lc{p}:=\lc{}\cap T_p{M}$:
\begin{enumerate}[label=(\roman*)]
  \item The set $\lc{p}$ is conic and salient, that is, for $v\in \lc{p}$, $\left\{ \lambda v:\lambda>0 \right\}\subset \lc{p}$ and $-v\notin \lc{p}$.
  \item $\lc{p}$ is the boundary in $T_p{M}$ of a non-empty open subset $A_p\subset T_p{M}$ which is convex, i.e., for any $w,z\in A_p$ the segment $\left\{ \lambda w + (1-\lambda)z:0\leq \lambda\leq 1 \right\}\subset A_p$.
  \item The second fundamental form of $\lc{p}$ regarded as an affine hypersurface of $T_p{M}$, is positive definite (with respect an inner direction pointing out to $A_p$) and its radical at each point $v\in \lc{p}$ is spanned by the radial direction. 
\item $\lc{}$ is transversal to the fibers of $TM$ (i.e. $T_v(T_pM)+T_v\lc{}=T_v(TM)$, for all $v\in \lc{p}$).
\end{enumerate}
\end{definition}

\begin{definition}
\label{def:main2:3}
A {\em Lorentz-Finsler metric} $L$ on a manifold with boundary $M$ is a   semi-Finsler metric $L:A\rightarrow \mathbb{R}$ such that each $L_p$ ($p\in M$) is a {\em (properly) Lorentz-Mikowski norm}, that is, $A_p$ is connected,  $L_p>0$, $L_p$ extends smoothly by $0$ to $\partial A_p$ and the fundamental tensor $g_v$ has Lorentzian signature $(+, -, \dots, -)$ at each $v\in \bar A_p$. A {\em Finsler spacetime} $(M,L)$ (or $(M,L,A)$) is a manifold $M$ endowed with a Lorentz-Finsler metric $L$. 
\end{definition}


The following precise connection between Finsler spacetimes and cone structures proven in \cite[Theorem 1.1]{JS20} for $\partial M=\emptyset$ extends directly to the case with boundary.

\begin{theorem}
\label{thm:main2:3}
Let ${M}$ be an $(n+1)$-dimensional manifold with boundary. 
\begin{enumerate}[label=(\roman*)]
    \item If $({M},L,A)$ is a   Finsler spacetime, the boundary $\partial A$ (i.e., the boundary of the topological closure of $A$ in $T{M} \setminus \left\{ \mathbf{0} \right\}$) defines a cone structure $\lc{}$ on $M$.
    \item Given a cone structure $\lc{}$ defined on ${M}$,  there is a Lorentz-Finsler metric compatible with $\lc{}$, that is, it is possible to choose $A \subset T{M} \setminus \left\{ \mathbf{0} \right\}$ and a Lorentz-Finsler metric $L: A \rightarrow \mathbb{R}$ so that $({M},L,A)$ is a Finsler spacetime with $\lc{} = \partial A$.
    \item The cone structures of two Lorentz-Finsler metrics $L_1, L_2$ coincide if and only if both metrics are anisotropically conformal, i.e., $L_2 = \mu L_1$, where $\mu > 0$ is defined on $\cls{A}$.
    \item The lightlike pregeodesics of two anisotropically conformal Lorentz-Finsler metrics  coincide.
    \end{enumerate}
\end{theorem}

\begin{remark} {\em (1) Any cone $\lc{}$,  defined as above, permits one to introduce the concepts of timelike, lightlike,  causal, and spacelike vectors, 
namely: a vector \(v \in T_p M \setminus{0} \) is timelike if \(v \in A_p\), lightlike if \(v \in \lc{p}(=\partial A_{p})\) causal if either timelike or lightlike and spacelike otherwise. Thus, any cone structure generalizes the relativistic notions corresponding to the {\em future-directed} cones of a time-oriented Lorentzian manifold.

(2) As $\lc{}$ can be determined by a Lorentz-Finsler metric $L$, the previous timelike vectors correspond with those $L(v)>0$, the lightlike ones with  $L(v)=0$ (in extended $L$), and the spacelike ones with the others in $TM\setminus\{0\}$, thus satisfying $L(v)<0$ close to $\lc{}$ in each extended $L$. These signs correspond to the natural use of signature $(+, -, \dots, -)$ for Lorentz-Minkowski norms, which  is the opposite  to  $(-, +, \dots, +)$, most commonly used in relativistic and semi-Riemannian literature, and then  for semi-Finsler metrics before. Anyway, this will be harmless, as the timelike/spacelike criterion respect to the cone $\lc{}$ is universally accepted. }
\end{remark}
From the previous theorem, cone structures do not uniquely determine a Lorentz-Finsler metric, but rather an entire class of anisotropically conformal metrics, extending the case of Lorentzian manifolds.
  However, the way to construct  Lorentz-Finsler metrics from a cone structure is much more flexible, as the following procedures shows.

\begin{remark}[Lorentz-Finsler metrics from cone triples]\label{r_conetriple} {\em The following procedure in \cite[Lemma 2.15, Theorem 2.17]{JS20} allows us to construct
naturally cone structures, including the case of a manifold with boundary $M$. 
Consider a \emph{cone triple} $(\Omega,\tmv{},F)$ composed by an 
 $1$-form $\Omega$, a vector field $\tmv{}$ and a Finsler metric $F$ on the kernel of $\Omega$, under the restriction $\Omega(\tmv{})\equiv 1$, and let $\Pi:TM\rightarrow$ ker$\Omega$ be the natural projection in the direction of $\tmv{}$. Then, $L:=\Omega^2-(F\circ \Pi)^2$ defines a Lorentz Finsler metric on $M$ with cone $\lc{} =\{\Omega(v)=F(v), v\in TM\{\setminus{0}\}$ with the caution that $L$ is smooth everywhere but in the direction of $\tmv{}$. Anyway, $L$ 
  can be smoothened in any small neighborhood around $\tmv{}$ (see \cite[Theorem 5.6]{JS20}), but we are interested only  in its value around the cone $\lc{}$, where it is smooth.
 Conversely, any cone structure $\lc{}$ permits to construct a (highly non-unique) cone triple $(\Omega,\tmv{},F)$ whose associated cone is $\lc{}$, just  choosing  as $\Omega$ any $\lc{}$-timelike form 
 (i.e. $\Omega(v)>0$ on any causal  $v$), as $\tmv{}$ any $\lc{}$-timelike vector with $\Omega(\tmv{})\equiv 1$ and $F$ determined by its indicatrix    (i.e., the set of its unit vectors), chosen  equal to $\Pi(\lc{} \cap \Omega^{-1}(1))$.
}
\end{remark}
  
\subsubsection{Basics on causality and boundaries}
Causality in Lorentzian manifolds with timelike boundary was studied systematically in \cite{RMI_AFS} for the case of Lorentz manifolds. Essentially, it can be transplanted directly to the case of Lorentz-Finsler manifolds (see \cite{JPS_Snell,Sa-BIRS}), anyway, the following observations are in order. First, the properties of causality depend only on the cone structure $\lc{}$  and we will refer undifferentiatedly to   $\lc{}$ or any compatible Lorentz-Finsler metric $L$ (or Finsler spacetime).

\begin{definition}
    A cone structure $\lc{}$ on a manifold with boundary $M$ has timelike boundary when the intersection $\lc{p}\cap T_p\partial M$ is non-empty and transverse for all $p\in \partial M$.
\end{definition}
\begin{remark}\label{r_timelike_no_degenerate}
    {\em For any compatible $L$, this means that the restriction of $L$ to $T\, \partial M$ is also a Lorentz-Finsler metric. In particular, $L$ has  non-degenerate boundary according to Definition~\ref{d_no_degenerate}. } 
\end{remark}

As in the case of Lorentzian manifolds, we can define the notion of chronologically (causally, horismotically) related points in any Lorentz-Finsler manifold. As emphasized in \cite[Section 2]{RMI_AFS} (see Remark \ref{r_cone_vs_horismos} below), an important caution is that the allowed regularity for the lighlike curves must be locally Lispchitz (or $H^1$), so, the causal character of their velocities is required only a. e.  Then, for $p,q\in \withboundaryM$,  $p$ is \emph{chronologically} (resp. \emph{causally};  \emph{horismotically}) \emph{related} to $q$, denoted $p\ll q$ (resp. $p\leq q$; $p\hookrightarrow  q$), if there exists a timelike (resp. causal; causal but not timelike) curve $\gamma:[a,b]\rightarrow \withboundaryM$ with $\gamma(a)=p$ and $\gamma(b)=q$. As $\lc{}$ plays the role of the future-directed cone in Lorentz spacetimes, one defines the usual chronological and causal future and pasts as
\begin{equation}
\label{eq:29}
\begin{array}{cc}
  I^{+}(p):=\left\{ q\in \withboundaryM:p\ll q \right\},& J^+(p):=\left\{ q\in \withboundaryM:p\leq q \right\}\\
    I^{-}(p):=\left\{ q\in \withboundaryM:q\ll p \right\},& J^-(p):=\left\{ q\in \withboundaryM:q\leq p \right\}.
\end{array}
\end{equation}
We will be interested in the intrinsic causality of the interior $\mathring{M}$, thus, the notation $I^{\pm}(p, \mathring{M}), J^{\pm}(p,\mathring{M})$ means that the corresponding subset is computed by using curves entirely  contained in $\mathring{M}$. Some basic properties, to be used next, are the following which extend \cite[Proposition 2.6]{RMI_AFS}.


\begin{proposition}\label{RMI}
 Let $(\withboundaryM,L,\withboundaryA)$ be a Lorentz-Finsler manifold with  timelike boundary. Then: (a)~$I^{\pm}(p)$ are open in $M$, (b) for any $p,q,r\in M$ with either $p\ll q\leq r$ or $p\leq q\ll r$, it follows that $p\ll r$, (c) $J^{\pm}(p)\subset \mathrm{cl}(I^{\pm}(p))$ and (d) $I^{\pm}(p,\mathring{M})=I^\pm(p)\cap \mathring{M}$.   In particular, the closure of $I^\pm(p,\mathring{M}$) in $\mathring{M}$ is equal to the intersection of $\mathring{M}$  with the closure of  $I^\pm(p)$ in $M$. 
\end{proposition}

\begin{proof} 
Properties (a), (b) and (c) can be obtained as in \cite[Props. 3.5, 3.6 and 3.7]{SolisPHD} by using \cite[Lemma 6.3]{AaJa}, which is a Lorentz-Finsler version of \cite[Lemma 10.45]{oneill}.\footnote{All properties can be also achieved by the corresponding Lorentz-Finsler adaptations of \cite[Lemmas 4.22 and 4.23]{FHS_ATMP}, where the timelike $C^1$-boundary is  identified with a part  of the causal boundary.}  
For (d), the proof follows in the same way as in \cite[Proposition 2.6 (d)]{RMI_AFS} but with a slight modification, namely, that the standard Lorentz choice of a normal vector to the boundary (whose extension to $M$ is denoted $N$ therein) can be weakened by  an arbitrary choice of a rigging vector field $\eta$. 
   \end{proof}
  
In the case of boundary, $\lc{}$ permits to distinguish  between the following  notions.
\begin{definition}
\label{def:main2:7} Let \(\lc{} \) be a cone structure on $M$  and  \(\gamma: I \rightarrow {M}\) be a curve (\(I\subset \R\)  an interval). 

(A)  \(\gamma\) is  \emph{locally horismotic} if it is locally Lipschitz (or $H^1$) and,  for each \(s_0 \in I\) and any open neighborhood \(V\) containing \(\gamma(s_0)\), there exists another open neighborhood \(U \subset V\), with \(\gamma(s_0) \in U\), that satisfies the following condition: given \(\epsilon > 0\) such that \(\gamma(I_{\epsilon}) \subset U\) where \(I_{\epsilon} = [s_0 - \epsilon, s_0 + \epsilon] \cap I\), and for any pair \(s_1, s_2 \in I_{\epsilon}\) with \(s_1 < s_2\), the points \(\gamma(s_1)\) and \(\gamma(s_2)\) can be connected by a lightlike curve within \(U\), but no timelike curve exists between them.

(B) \(\gamma\) is a  \emph{pre-geodesic for $\lc{}$} if it is smooth and satisfies that it is a lightlike geodesic for a compatible Lorentz-Finsler metric, up to a reparametrization.
\end{definition}

\begin{remark}\label{r_cone_vs_horismos} {\em In case $\partial M=\emptyset$ the pre-geodesics for $\lc{}$ are equal to the lightlike (pre-)geodesics of any Lorentz-Finsler metric $L$ compatible with $\lc{}$ by  \cite[Theorem 6.6]{JS20} (indeed, both terms are used interchangeably therein). However, an example in \cite[Appendix B]{RMI_AFS} shows that, even in the case of a $C^\infty$ boundary $\partial M$, the locally Lispchitz curves (or, under a more general viewpoint, the $H^1$ ones) permits to reach  a points horismotically related which cannot be connected by means of a (piecewise) smooth curve. Anyway, locally horismotic curves are a.e. differentiable (by Rademacher theorem) and in these points, the velocity must be lightlike.}
\end{remark}

\subsection{Lightconvexity of $\partial M$ and causal simplicity for $\mathring{M}$}
\label{sec:caus-simpl-conv}

\subsubsection{Relevant steps of the causal ladder}
The following steps of the ladder of causality of spacetimes-with-timelike-boundary will be used along this paper. Their definitions are formally equal  in the cases with and without boundary, as well as in the cases Lorentz and Lorentz-Finsler.  
Namely, a cone structure (or Lorentz-Finsler metric  $L$) on a manifold  with boundary $M$ is: 
\begin{itemize}
\item {\em Causal} if it admits no closed causal curve.
\item {\em Strongly causal} if it is {strongly causal  at each  $p \in M$}, the latter meaning that,  for any neighborhood $V$ of $p$ there exists a smaller one $U\subset V$ such that any  timelike curve starting at $U$  that  leaves $V$  (both towards the future  or  the past)  does not enter $U$ again.     
\item {\em Causally simple} if it is causal and the sets $J^{\pm}(p)$ are closed for all $p\in \withboundaryM$.
\item {\em Globally hyperbolic} if it is causal and the intersections $J^+(p)\cap J^-(q)$ are compact for any pair $p,q\in \withboundaryM$. 
\end{itemize}

The  properties of the case with boundary were analyzed in the Lorentz case  in \cite{RMI_AFS} and we will be interested in several ones that  extended directly to the case Lorentz-Finsler.
First, each level implies the previous one \cite[Lemma 3.6]{RMI_AFS}. Second, neither the global hyperbolicity nor the causal simplicity of $M$ imply the causal simplicity of its interior $\mathring{M}$, but only a weaker property (namely, $\mathring{M}$ is causally continuous \cite[Theorem 3.8]{RMI_AFS}).  

\subsubsection{Main result}
Next, our aim is to prove that 
lightconvexity of the boundary characterizes when the interior is causally simple.


\begin{theorem}
\label{thm:main2:4} 
Let $(\withboundaryM,L,{\withboundaryA})$ be a causally simple Lorentz-Finsler manifold with timelike boundary. Then, they are equivalent:

(i) $\partial M$ is infinitesimally lightconvex. 

(ii) $\partial M$ is locally lightconvex. 

(iii) The interior $\mathring{M}$ of $M$ is causally simple. 
\end{theorem}

\begin{proof}
 (i) $\Longleftrightarrow$ (ii) Notice that Theorem \ref{t1} is aplicable by Remark \ref{r_timelike_no_degenerate}.

(iii) $\Longrightarrow$ (i). Assume by contradiction 
$II^w(w,w)<0$ with 
$w=\dot\gamma(0) \in \lc{p}$, as 
in Proposition \ref{prop_strictconv}. Take a  small globally hyperbolic  neighborhood $\widehat{U}$ of $p$ in $\widehat{M}$ so that the  geodesic $\gamma$ is included in $\mathring{M}\cap \widehat{U}$, up to $\gamma(0)=p$. Then, for some small $\delta>0$ the points 
$p_\pm :=\gamma(\pm \delta)$ are horismotically related in $\widehat{M}$ and the lightlike geodesic $\gamma$ is 
the unique causal curve in $\widehat{M}$ connecting them. Thus, $p_-\not\in J^-(p_+, \mathring{M})$. However,
$p_-$ belongs to the closure of $I^-(p_+, \mathring{M})$ in $\mathring{M}$  by the last assertion of Proposition  \ref{RMI}, contradicting the causal simplicity of $\mathring{M}$.

(ii) $\Longrightarrow$ (iii)
  Let us consider $p\in \mathring{M}$ and a point $q$ 
  in the closure of 
  $I^{+}(p, \mathring{M})$ computed in $\mathring{M}$ We have to prove $q\in J^{+}(p, \mathring{M}$) (a dual reasoning would work for $I^-, J^-$). 
  
  From the causal simplicity   of $M$,  $q\in J^+(p)$, thus,  there exists a (locally Lipschitz) causal curve $\gamma:[a,b]\rightarrow \withboundaryM$ from $p$ to $q$. 
  If $\gamma\subset \mathring{M}$ then we are done. Moreover,  this also happens if $p\in I^-( q)$ by Proposition \ref{RMI} (d).  So,  $p$ and $q$ must be joined  by a locally horismotic curve 
  $\gamma$. 
  
  If $\gamma$ is smooth, then it is a lightlike geodesic, it must intersect the boundary and must be tangential therein. 
  Hence, it should be contained in $\partial M$ by our hypothesis (ii), a contradiction. 
  
  Otherwise, $\gamma$ can be parametrized as a Lipschitz curve in any local chart\footnote{Such a parametrization (equivalent here to $H^1$) will be used below to ensure the existence of convergent subsequences for $\gamma'$. Moreover, one can even assume that $\gamma$ is globally  parametrized  by a temporal function $\tau$, which is a common choice for several purposes (see  \cite[Corollary 4.13]{JPS}, where the equations of lightlike pregeodesics under such a parametrization are provided).} with $0\in (a,b)$ in such a way that  $r:=\gamma(0)$ is the first intersection point of $\gamma$ and $\partial M$ from $p$. 
  Thus, $\gamma|_{[a,0]}$ is smooth\footnote{\label{f_RMI} Taking into account that the chronological future in $M$ can be computed by using locally Lipschitz curves as well as piecewise ones (by \cite[Section 2.4]{RMI_AFS}, which is directly extensible to the Lorentz-Finsler case)}, in particular, 
  $v_- := 
  \lim_{t\nearrow 0} \dot{\gamma}(t)$ is well defined  and it cannot be tangent to $\partial M$, moreover, $v_-$ must point outwards $\partial M$.    Then, 
choose a sequence  $t_m\searrow 0$ where $\gamma$ is differentiable, and take  the limit (up to a subsequence) $v_+:=\lim_{t_m\searrow 0} \dot\gamma(t_m)$. Obviously $v_+$ cannot point out outwards. Thus, 
$$\Delta v:=v_+-v_-$$ 
points out inwards. This property will permit to find a contradiction, namely, a   variation of $\gamma$  by  causal curves included in $\mathring{M}$ with fixed endpoints and positive length. For this purpose we will extend  arguments as   \cite[Prop. 10.46]{oneill} for the regular Lorentzian case without boundary. 

Let $W$ be the vector field along $\gamma$ obtained by propagating parallelly $\Delta v$ by using the {\em Levi-Civita-Chern} anisotropic connection $\nablaLCC$ (see Appendix \S \ref{A_anisotropic}) in the direction of
\footnote{In the Lorentz case, the existence and uniqueness of $V$ is ensured because the equation of the parallel transport is a first order ODE system with Lipschitz coefficients in each coordinate chart. This can be extended to the Lorentz-Finsler case, as the expression of coordinates gives a first order equation of the same type, see \cite[eqn (32)]{Javaloyes_2022}}
$\gamma'$. At 0 one has $g_{v_\pm} (v_\pm, \mp \Delta v)>0$ and, by continuity,  $g_{\gamma'}(\gamma',W)$ 
is positive in  $[a',0]$ and negative on $[0,b']$ for some   $a\leq a'<0<b'\leq b$; we can assume $a=a', b=b'$ with no loss of generality\footnote{Indeed, as $\gamma$ is a geodesic on $[a,0]$ and $W$ is parallel then $g_{\gamma'}(\gamma',W)$ is constant therein (as $\nablaLCC$ parallelizes the fundamental tensor $g$, see \cite[Proposition 7]{Javaloyes_2022}). However, this does not hold on $[0,b]$ and only the continuity of the parallel transport is claimed therein.}. 
Now choose a piecewise smooth function f on $[a, b]$ that vanishes at endpoints and has derivative $f’$ positive on $[a, 0]$ and negative
on $[0,b]$. Then, for $V = f W$ we have $g_{\gamma'}(V’, \gamma ’) > 0$. Easily, the  symmetry of the Levi-Civita-Chern anisotropic connection  implies that the longitudinal curves are timelike (see \cite[Lemma 6.3]{AaJa}, which extends the Lorentz case \cite[Lemma 10.45]{oneill}). 

\end{proof}

\begin{remark}
   {\em  In a globally hyperbolic spacetime-with-timelike-boundary, the interior $\mathring{M}$ cannot be globally hyperbolic as long as $\partial M\neq \emptyset $ (indeed, in a natural way $\partial M$ can be identified with the {\em naked singularities} \cite[Appendix A]{RMI_AFS}). Thus, causal simplicity is the best intrinsic step of the causal ladder for both the hypothesis and the item (iii) of the theorem.}
\end{remark}

\section{The lightspace $\Ng$  of  the cone geodesics 
 in a Finsler  spacetime}
\label{s_4}

\subsection{$\Ng$ for anisotropic cone structures}\label{s_4.1}

The space of  lightlike (or {\em null}) geodesics for a (Lorentz) spacetime was introduced by Low in \cite{Low_1989} (see also \cite{Low_2006}). Next, we will consider its generalization, the  {\em space of cone geodesics} or simply {\em lightspace} for  $\Ng$ any cone structure $\lc{}$ (equally, for any compatible Lorentz Finsler $L$). It is worth pointing out:
(a) Low's construction  extends in a natural way to our general cone structures $\lc{}$, even if the latter involves nonlinear connections, (b) we will be interested in the space of cone geodesics $\Ng$
for the interior $\mathring{M}$ of a Finsler spacetime with boundary $M$ (and how the boundary 
$\partial M$ provides information on this space), anyway, the basic general considerations in this subsection are valid also for the case with boundary, and (c) we will use a construction in the tangent bundle $TM$ which is enough for our purposes, even if it is less refined than  Low's one, who used a Hamiltonian approach in the cotangent bundle $T^*M$.


\subsubsection{$\Ng$ as a quotient} In the Lorentz-Finsler case, as in the Lorentz one, both the geodesic spray $G$ and the radial (Liouville) field $\mathcal{R}:=y^k \frac{\partial}{\partial {y^k}}$ preserve the tangent to the cone $T\lc{}$. Moreover,  from the positive two homogeneity of the functions $G^a$ in Definition \ref{def:main2:13} 
 one obtains,   
\begin{equation}\label{e_corchete}[\mathcal{R},G]=G. \end{equation}
 Thus, Span $\{\mathcal{R},G\}$ is involutive, supporting the following definition.\footnote{More generally,  for 2-positive homogeneous vector fields $G$ on $TM$, it is well known that $G$ is a geodesic vector field  for a symmetric nonlinear connection if and only if \eqref{e_corchete} holds and $\nu(G)=R$, where the endomorphism $\nu: 
T(TM)\rightarrow T(TM)$ is defined by taking  
$\nu(w_{v_p})$ equal to the vertical vector in 
$T_{v_p}(TM)$ naturally identifiable to 
$d\pi_{v_p}(w_{v_p})\in T_pM$ (here $\pi: TM\rightarrow M$ is the natural projection), for all 
$p\in M, v_p\in T_p M, w_{v_p}\in T_{v_p}(TM)$ (see for example \cite[Theorem 7.3.4]{Szilasi2012}).}
 

\begin{definition}
    The space of cone geodesics $\Ng$ is the space of leaves $T\lc{}/\mathcal{F}$, where $\mathcal{F}$ is the foliation determined by the integrable distribution {\rm Span} $\{\mathcal{R},G\}$ on $T\lc{}$, endowed with the quotient topology.
\end{definition}

\begin{remark}{\em 
Low's  Hamiltonian  approach is also extensible to the Finsler case. Indeed, each lightlike vector $v$ can be identified with its Legendre
transform $g_v(v,\cdot)$ for $L$. The set of all these forms is then the {\em cotangent null space} $\Ng^*M$. This space is preserved by the flow of the Euler vector field $\Delta$, and its quotient has a natural interpretation: all such forms obtained along a lightlike direction $\R^+ v$, $v\in T_pM$ for some $p \in M$, are identified with their common kernel, that is, the hyperplane $H_{\R^+ v}$ tangent  to 
$\lc{p}$ at the direction $\R^+ v$. 
The geodesic vector field $X_G$ associated with the Hamiltonian corresponding to $L$ also preserves 
$\Ng^*M$ and it satisfies $[\Delta,X_G]=X_G$. Then,  our space of cone geodesics $\Ng$ becomes equivalent to Low's quotient $\Ng^*M/\mathcal{F^*}$, where $\mathcal{F^*}$ is the foliation of $\Ng^*M$ determined by the integrable distribution Span$\{\Delta,X_G\}$.
}\end{remark}

\subsubsection{Differentiable structure of $\Ng$ in the strongly causal case}\label{subsub_dif_est_n}

In general, a foliation $\mathbf{F}$ on a manifold $Q$ 
is called {\em simple}  if, for each $p \in Q$, there exists a foliated chart $(U, \phi)$  with
the property that every leaf $L$ intersects $U$ at most in one plaque. It is well-known that 
 a foliation $\mathbf{F}$ is simple if and only if there exists a (possibly non-Hausdorff) smooth structure on the quotient $Q/\mathbf{F}$ for which the natural projection $\Pi: Q\rightarrow Q/\mathbf{F}$ is a submersion. 
\footnote{Indeed,  classical references as \cite{molino} define simple foliations as those obtained from submersions $f:M\rightarrow N$, where the leaves are the corresponding connected fibres of $f$, being the equivalence well-known
\cite[Page 70]{Ruiloja}.   
} 
 
\begin{proposition}\label{prop:manifoldfoliation}
    When the cone structure $\lc{}$ on $M^{n+1}$ is strongly causal, the foliation 
    $\mathbf{F}$ integrating  {\rm Span} $\{\mathcal{R},G\}$ is simple and, thus, 
         $\Ng$ is a possibly non-Hausdorff (2n-1)-manifold. 
\end{proposition}
\begin{proof}
    The first assertion is straightforward (reason as in \cite[Prop. 2.1]{Low_1989}  or \cite[Th. 1]{Low_2006}) and  
the required dimension follows because 
 $\lc{}$
has dimension $2n+1$  and  the leaves of $\mathbf{F}$ dimension 2.
\end{proof}

In what follows, we will consider  strongly causal cone structures  
$\lc{}$ 
corresponding to the interior $\mathring{M}$, and the following illustrative representation of $\Ng$ emerges naturally: 
\begin{itemize}
    \item As a point set, each element of $\Ng$, which will be called a \emph{cone geodesic}, 
    is  a set of  {\em inextendible lightlike pregeodesic}, namely, the set of curves obtained taking an inextendible  lightlike geodesic\footnote{As $\gamma$ is included in the interior $\mathring{M}$, the inextensibility of $\gamma$ as a geodesic beyond   any of the extremes $c$ of $I$
is equivalent  to the inexistence in $\mathring{M}$ of  $\lim_{t\rightarrow c} \gamma(t)$; however, such a limit might exist in $\partial M$.} $\gamma: I\rightarrow \mathring{M}$ of $L$ and all its (oriented) reparametrizations. Inextendible lighthlike pregeodesics can also be regarded as embedded degenerate totally 
geodesic\footnote{Recall that, for degenerate submanifolds, totally geodesic means that any  geodesic in the ambient with initial velocity in the submanifold remains in it.}  1-dim. submanifolds of $\mathring{M}$ with no endpoints in $\mathring{M}$. { In the remainder, cone geodesics and any of the pregeodesics  constituting them will be denoted 
equally (and distinguished by context). }
\item The topology of $\Ng$ is provided by the following natural limit
operator\footnote{See \cite[Appendix]{FHS_ATMP} and \cite[Section 2.1]{AkeSpacetimecoveringscasual2017} for background on topologies determined by limit operators.}:  
a sequence of cone geodesics $\{\gamma_k\}_k$ converges to $\gamma_\infty$
when they admit
 such that
$\{\gamma'_k(0)\}
\rightarrow
\gamma'(0)$.
This topology is {\em sequential} (that is, any subset $C\subset \Ng$ is closed when it satisfies: if a sequence in $C$ converges to some point $x \in \Ng$, then $x\in C$) by \cite[Prop 3.39]{FHS_ATMP} and, thus, it is fully determined by the convergence of 
sequences.

\end{itemize}


\subsubsection{The issue of Hausdorffness}

It is not surprising  that, in strongly causal spacetimes, $\Ng$ may be non-Hausdorff  even if it has a smooth structure. 
Indeed, non-Hausdorffness happens  in the Riemannian case for the space of geodesics of $\R^n\setminus\{0\}$ (in the sense of convergence above)  and, then, this is  transmitted to $\Ng$  in the natural Lorentzian product $\Lo^1\times (\R^n\setminus\{0\})$. 
In general, Hausdorffness is characterized by the following elementary result \cite[p. 69]{Ruiloja}.
\begin{lemma}\label{lemma_Rui}
Let $X$ be a Hausdorff topological space and $\sim$ an equivalence relation in $X$
such that $\Pi : X \rightarrow X/ \sim$ is an open map for the quotient topology on $X/ \sim$. Then:
$X/ \sim$  is Hausdorff if and only if the graph of $\sim$ is
closed in $X \times X$.
\end{lemma}
In the  strongly causal case, the projection $T\lc{}\rightarrow \Ng$ is a submersion among manifolds (the first one Hausdorff) and, by the theorem of the implicit function, it is open. Therefore:
\begin{proposition}\label{p_Rui}
    For strongly causal cone structures, $\Ng$ is Hausdorff if and only if  the graph of the relation of equivalence $\sim$ in $T\lc{}$ which defines $\Ng$ is
closed in $T\lc{} \times T\lc{}$.
\end{proposition} 
We will come back to this  in \S $\ref{s_last}$.  
Now, notice the following first criterion of non-Hausdorffness.
\begin{proposition}\label{p_noHaus}
If a Finsler spacetime with timelike boundary is   strongly causal at $p\in \partial M$ and   $II^w(w,w)<0$ holds for some $w\in \lc{p}$, then $\Ng$ is not Hausdorff.
\end{proposition}

\begin{proof}
    Consider the geodesic $\gamma$ with $w=\dot \gamma(0)$ as in the proofs of Proposition \ref{prop_strictconv} and Theorem~\ref{thm:main2:4} (second part), defined in a small adapted neighborhood  $U$ of $p$ such that all the inextensible causal curves abandon $U$ and do not come back $U$ again.  Then $\gamma\setminus \{p\}$ yields two null inextensible geodesics in $\mathring{M}$ which are distinct because of the properties of $U$, and any sequence of lightlike geodesics $\{\gamma_k\}$ in $U$ with $\lim_{k\rightarrow \infty}\dot \gamma_k(0)=\dot \gamma(0)$ shows non-Hausdorfness.
\end{proof}

\begin{proposition}\label{p_Haus}
  $\Ng$ is Hausdorff for any lightconvex  globally hyperbolic Finsler spacetime-with-timelike-boundary. 
\end{proposition}

\begin{proof}
Assume by contradiction that there exists a sequence of 
 cone  geodesics $\{\gamma_m\}$, $m>3$,  converging to other two  $\gamma_1, \gamma_2$. Let $t_m\rightarrow t_\infty, s_m\rightarrow s_\infty, t_\infty < s_\infty$ such that $\gamma'_{m}(t_m) \rightarrow \gamma'_1(t_\infty), \gamma'_{m}(s_m) \rightarrow \gamma'_2(s_\infty)$ and choose $p,q\in M$ with $p\ll \gamma_1(t_\infty), \gamma_2(s_\infty) \ll q$. Then, 
$\gamma_m(t_m), \gamma_m(s_m) \in J(p,q)$  for $m$ big enough (as $p\ll \gamma_m(t_m)$, $\gamma_m(s_m)\ll q$ and $\gamma_m(t_m)\leq \gamma_m(s_m)$)  and the piece of $\gamma_1$ starting at $t_\infty$ (resp. of $\gamma_2$ ending at $s_\infty$) 
is included in the compact subset $J(p,q)$ of $M$.  Necessarily, $\gamma_1$ has an endpoint $p_1$ at $\partial M$ and  $\gamma_2$ an initial point $p_2 \in \partial M$, as both are inextensible in $\mathring{M}$. The lightconvexity of $\partial M$ makes   $\gamma'_1$ to point out  outwards at $p_1$ (and $\gamma'_2$ inwards  at $p_2$). 
Thus, the uniform convergence on compact subsets of the sequence (on an extension $\widehat{M}$ of the spacetime $M$ beyond, say, $p_1$) makes $\gamma_m$ to arrive $\partial M$ for large $m$ too, a contradiction. 
\end{proof}

\subsection{Structure of $\Ng$ for globally hyp. Finsler spacetimes-with-timelike-boundary}
\label{s_4.2}
It is known that the space of 
cone geodesics $\Ng$ 
of a (Lorentz) globally hyperbolic $(n+1)$-spacetime without boundary has a natural structure of smooth manifold of dimension $2n-1$ obtained as a projectivization of the cone bundle of any of its smooth spacelike Cauchy hypersufaces $S$ \cite{Low_1989} (the existence of such a $S$ ensured by \cite{BS03}). Moreover, a natural representation of this space can be obtained as the unit sphere bundle $U S$ 
 as follows. Take the future-directed unit timelike vector normal to $S$. Each 
future directed lightlike geodesic $\gamma$ can be identified with the unique vector $v_p\in U S$ such that, 
when $\gamma$ crosses $S$ at a point $p$, its velocity is proportional to $\tmv{p}+v_p$.  

This can be extended directly to the Lorentz-Finsler setting as it is known that they also admit smooth spacelike Cauchy hypersurfaces\footnote{\label{foot} This was done by Fathi and Siconolfi \cite{Fathi2011}, including the existence of a Cauchy temporal function. The possibility to extend the original arguments by Bernal and S\'anchez \cite{BS03, BS05, BS07} from the Lorentzian case to the Lorentz–Finsler one was first outlined in \cite{Sa_Penrose}, and a full proof (including the extension to the case with boundary in \cite{RMI_AFS}) can be found in \cite[Theorem 3.3]{Sa-BIRS}.}. 
In fact,  consider a Finsler spacetime without boundary and recall the  construction with a triple $(\Omega,\tmv{},F)$ for the associated cone structure in Remark \ref{r_conetriple}.

\begin{theorem}\label{t_4.4}
Let $(M,L,A)$ be a globally hyperbolic Finsler ($n+1$)-spacetime without boundary.
Then, its space of lightlike geodesics $\Ng$
has a natural structure of smooth manifold of dimension $2n-1$ obtained as a projectivization of the cone bundle on any of its (smooth) spacelike Cauchy hypersufaces $S$.  
Thus, $\Ng$ is naturally identifiable with the unit sphere bundle (indicatrix) for a Finsler metric $F$ on 
$S$. 
\end{theorem}
\begin{proof}
As $S$ is spacelike, at each $p\in S$ there is a unique $1$-form $\Omega_p$ and a unique timelike unit vector $\tmv{p}$ such that $\mathrm{Ker}\Omega_p=T_pS$ and $\Omega_p(\tmv{p})=1$. Indeed, $\tmv{p}$ is the tangency point at $L_p^{-1}(1)$ of a unique  affine hyperplane  in $T_pM$ parallel to $T_pS$, and this hyperplane  becomes $\Omega_p^{-1}(1)$. The intersection of $\Omega_p^{-1}(1)$   
with the ligthlike cone $\lc{p}$ 
is a smooth compact strongly convex hypersurface of $\lc{p}$  whose projection on $T_pS$ in the direction of $\tmv{p}$ gives the required indicatrix for a norm $F_p$ on $T_pS$ and, thus, the required Finsler metric $F$ (see also \cite[Theorem 2.17]{JS20}).
\end{proof}

In order to extend this result to the case with timelike boundary, the global splitting for  the Lorentzian case \cite{RMI_AFS} (which can be extended to the Lorentz-Finsler case, see footnote \ref{foot}) will be used. Indeed, this splitting will permit to introduce  a parametrization of  both $\Ng$ and an appropriate smooth Hausdorff (2n-1)-manifold (without boundary), providing atlases for both spaces and  a natural identification between them. 

\begin{convention}\label{convention} {\em From the discussion above, we can assume that any globally hyperbolic Finsler spacetime-with-timelike-boundary is written as a product
\begin{equation}\label{t_globalsplitting}
M=I\times S_0, \qquad \hbox{where} \; I\subset \R \; \hbox{is an interval  } 
\end{equation} where  the natural projection $t=I\times S_0\rightarrow I$ has levels $t=$ constant which are  
spacelike Cauchy hypersurface  (with boundary). 
 One can assume  $I\equiv  \R$  by rescaling $t$ (which would be then called a {\em Cauchy temporal function}), however, the above generality is useful in specific  examples. Once this splitting is prescribed, $S$ will denote a (possibly different) choice of spacelike Cauchy hypersurface. In any case, $\partial M$ becomes diffeomorphic to $\R \times \partial S$.}
\end{convention}
 \begin{definition}\label{d_ll}
 For any spacelike Cauchy hypersurface $S$ consider  the following open subsets of  the  projectivizations of the cones $\lc{}$ and $-\lc{}$ restricted to some hypersurfaces, endowed with their natural topologies:
\begin{itemize}
\item $\ell^+_S:= \ell_{\partial S}^+ \cup \ell_{\mathring{S}}^+$, where
\begin{itemize}
 \item $\ell_{\mathring{S}}^+:=\{\R^+\cdot v_p:  v_p\in \lc{p}  \;  \hbox{and} \;  p\in \mathring{S}\}, $
 
 \item $\ell_{\partial S}^+:=\{\R^+\cdot v_p:  v_p\in \lc{p}, \, \hbox{points inwards}  \; \partial M \, \hbox{and} \;  p\in J^+(\withboundaryS)\cap \partial M\}, $

\end{itemize}
\item $\ell^-_S:= \ell_{\partial S}^- \cup \ell_{\mathring{S}}^-$, where
\begin{itemize}
 
 \item $\ell_{\mathring{S}}^-:=\{\R^+\cdot v_p:  -v_p  \in \lc{p} \;  \hbox{and} \;  p\in \mathring{S}\}, $

 \item $\ell_{\partial S}^-:=\{\R^+\cdot v_p:  -v_p\in \lc{p}, \; \hbox{points outwards}  \, \partial M \, \hbox{and} \;  p\in J^-(\withboundaryS)\cap \partial M\}. $   
 
\end{itemize}
\end{itemize}
Moreover, consider the open subsets (independent of the choice of $S$) 
\begin{itemize}
\item $ \ell_{\partial M}^+:=\{\R^+\cdot v_p:  v_p\in \lc{p}, \, \hbox{points inwards}  \; \partial M \, \hbox{and} \;  p\in \partial M\} 
    . $
\item $\ell_{\partial M}^-:=\{\R^+\cdot v_p:  -v_p\in \lc{p}, \; \hbox{points outwards}  \, \partial M \, \hbox{and} \;  p\in \partial M\}$
\end{itemize}    
 \end{definition}
The manifold structure of $\ell_{\mathring{S}}^\pm$ as a sphere bundle on  $\mathring{S}$ is well known. Next, we consider the other pieces $\ell_{\partial S}^\pm$ of $\ell^\pm_S$ as well as $\ell_{\partial M}^\pm$.
\begin{proposition}\label{p_ell}  (1) $\ell_{\partial M}^+$ (resp. $\ell_{\partial M}^-$) is diffeomorphic to $\R\times T\, \partial S$.


(2)    $\ell_{\partial S}^+$ (resp. $\ell_{\partial S}^-$) is diffeomorphic to $[0,\infty) \times T\, \partial S$.
\end{proposition}
\begin{proof}          
(1) Assume  $I=\R$, $\partial M=\R\times \partial S$  as in Convention \ref{convention}. For each $p=(t,x)\in \partial M$, the  vectors  $v_p\in \lc{p}$ 
     pointing inwards project naturally onto inner pointing vectors of $\{t\}\times T_x S$.  Thus, the inwards directions of $\lc{p}$  turns out naturally diffeomorphic to the projectivization of 
    the  directions of $T_x S$ inner pointing respect to $\{t\} \times T_x\partial S$ (topologically an  open half-sphere\footnote{Choosing a cone triple $(\Omega,  \mathbb{T}, F)$ with $\Omega=dt$ and $\mathbb{T}$ tangent to $\partial M$, 
      those vectors with $dt(v_p)=1$ project on an open  half of the indicatrix of $F$ at $p$.} 
    $\mathbb{S}^{n-1}_+$, and just one point in the limit case $n=1$). The latter is naturally diffeomorphic to $T_x \partial S$. In fact, 
    choosing $\eta_x \in T_xS$   transverse to $\partial S$ and inner pointing, one has the natural identification
$$ T_x\partial S \ni \quad u_x \; \mapsto \; \R^+ \cdot (u_x +\eta_x) \quad \in \{ \hbox{inner pointing directions of} \;  T_xS
\}
$$    
This construction at each $p$ can be globalized  as an inner pointing vector field  $\eta$ transverse to $\partial S$ can be chosen and extended  to the product $\R\times \partial S$, providing then  a diffeomorphism between the inner pointing directions of $\R^+ \cdot  \lc{}$ in $\partial M$  and $\R\times T\, \partial S$.

(2) Any spacelike Cauchy hypersurface $S$ can be regarded as the level $t=0$ of a temporal function $t$ (see \cite{BS07} and footnote \ref{foot}) and, thus, $M$ splits as in Convention $\ref{convention}$ with $S=S_0$. Thus $J^+(S)\cap \partial M$ becomes difeomorphic to the product $[0,\infty)\times S$ and the proof follows as in  case~(1).
\end{proof}

Next, we will see that  $\ell^\pm_S$ are homeomorphic to two open subsets $\Ng^\pm_{S}$  (depending on the Cauchy hypersurface $S$) of $\Ng$ 
and the union  of these two pieces of $\ell^+_S$ (resp. $\ell^-_S$) 
also inherits a natural differentiable structure.  
\begin{lemma}\label{lemma:hypersurface} For any  globally hyperbolic Finsler spacetime $(\withboundaryM,L,\withboundaryA)$ with lightconvex  timelike-boundary, there is a  natural bijection between $\ell^+_S$ (resp. $\ell^-_S$)
 and the open subset $\Ng_S^+$ (resp. $\Ng_S^-$) of the lightspace $\Ng$  
of $\mathring{M}$ whose cone geodesics have  at least one point at $I^+(S)$ (resp. $I^-(S)$).  Moreover, $\ell^+_S$ (resp. $\ell^-_S$)
admits  a natural smooth differentiable structure as a 
(Hausdorff) manifold of dimension  $(2n-1)$ diffeomorphic to $\Ng_S^+$ (resp. $\Ng_S^-$).     
\end{lemma} 

\begin{proof}
Reasoning for \(\ell^+_S\) (for \(\ell^-_S\) would be analogous), trivially each $\R^+\cdot v_p\in \ell^+_S$, determines a unique cone geodesic in $M$, and it intersects $\mathring{M}\cap I^+(S)$. 
Conversely, if $\gamma$ is a future-directed lightlike geodesic of $L$ and, at some point, $\gamma(t_0) \in \mathring{M}\cap I^+(S)$,  going to the past along $\gamma$ it will  cross 
(by the global hyperbolicity of $M$) either $\mathring{S}$, thus 
selecting a  direction in $\ell^+_{\mathring{S}}$ or $\partial M$, in the latter case transversally (by the lightconvexity of $\partial M$), thus 
selecting a  direction of $\ell^+_{ \partial S }$. This provides the required bijection with $\Ng_S^+$, which is clearly an open subset of $\Ng$.

The previous bijection and the existence of a structure of (Hausdorff) manifold for $\Ng$  (Propositions \ref{prop:manifoldfoliation}, \ref{p_Haus})  and, thus, on $\Ng_S^+$, would permit to induce a diffeomorphic structure on $\ell^+_S$. Next, we will show explicitly this structure, checking how it becomes natural for $\ell^+_S$ too. 

Clearly, \(\ell_{ \partial S }^+\) and \(\ell_{\mathring{S}}^+\) are disjoint smooth \((2n-1)\) submanifolds of the projectivization of $TM\setminus \{0\}$,  \(\ell_{ \partial S }^+\) having a boundary. Moreover,  \(\ell_{\mathring{S}}^+\) naturally extends to the points of $\partial S$ as a manifold with 
boundary\footnote{Crucially, for these points 
 all the directions non-pointing out inwards are removed  using lightconvexity. As these removed directions constitute a closed subset in the set of directions, its exclusion does not affect the differentiable structure. Certainly, it does affect to the bundle structure, whose fibers are  naturally a sphere in $\mathring{S}$ but an open half sphere on $\partial S$ (thus even topologically different when $n=1$). The fiber structure is not relevant in this proof, but it helps to understand the whole construction.} too and  matches with \(\ell_{ \partial S}^+ \) therein.  Thus, in principle, the natural identification of the boundary points of both manifolds results in a piecewise smooth manifold  (i.e., a manifold with corners). However, these corners are not natural from the viewpoint of the cone geodesics (indeed, the cornes would disappear by choosing a  Cauchy hypersurface different to $S$), and the following natural smooth reparametrization at the corner points of $\ell^+_S$ emerges naturally.

For   $p\in \partial S$ and $ v_p \in  T    \withboundaryM  \cap \lc{p}$    transversal to $\partial M$  pointing inwards,    consider an open 
neighborhood \(U \subset T\withboundaryM\setminus \{0\} \) around \(v_p\) such that every \( w_q  \in U \cap \lc{q}\) at \(q \in \partial M\) is also transversal  to $\partial M$.   Let \(\ell^+_{\partial M}\) as in Definition \ref{d_ll}, and make also \(U\)  smaller so that \(U_{\partial M}:= \ell^+_{\partial M} \cap U\) admits a smooth chart as an open subset of $\ell^+_{\partial M}$.  
Now,  by taking $U$ sufficiently small\footnote{In order to ensure that the lightlike geodesic associated with  $w_q$ does not intersects again the boundary avoiding to reach $\mathring{S}$.}, any element \(w_q \in U_{\partial M}\) with $q\in I^-(S)$ defines a 
lightlike geodesic that intersects \(\mathring{S}\) at some $q'$, mapping the previous \(w_q\) to an element 
\(w'_{q'} \in \ell^+_{\mathring{S}}\). 
In this way, a homeomorphism between both neighborhoods of $w_p$,  $U_{\partial M}$ in $\ell^+_{\partial M}$ and an open subset of the corner  \(\ell^+_{\partial M} \cup \ell^+_{\mathring{S}}\) is obtained by using the geodesic flow. Thus the differentiable structure of the former (provided by a chart for $U_{\partial M}$ in $\ell^+_{\partial M}$, which is transversal to the plaques of the geodesic foliation) yields the required coordinates free corners. 
\end{proof}
Now, we can  prove the main result on this section.
\begin{theorem}\label{t_Ng_estructuraglobal}
Let $(\withboundaryM^{n+1},L,\withboundaryA)$ be a globally hyperbolic  Finsler spacetime-with-timelike-boundary such that $\partial M$ is infinitesimally lightconvex. Then, its interior is causally simple and the space of its cone geodesics is a  smooth Hausdorff (2n-1)-manifold which can be described as follows. If $\withboundaryS$ is any spacelike Cauchy  hypersurface (with boundary) then:  

\begin{enumerate}[label=(\roman*)]
\item \label{item1thmMiguel} The space $\Ng_S^+$  (resp. $\Ng_S^-$)  of all the cone  geodesics in $\mathring{M}$ reaching $I^+(S)$ (resp. $I^-(S)$) has the structure of a smooth differentiable manifold homeomorphic to $\ell^+_S$ (resp. $\ell^-_S$) and it is diffeomorphic  when the latter is endowed with the differentiable structure described  in the proof of Lemma \ref{lemma:hypersurface}.

\item \label{item2thmMiguel} The space $\Ng$ of all the cone  geodesics  in  $\mathring{M}$  has the structure of a smooth differentiable manifold naturally identifiable to $(\ell^+ \cup \ell^-)/\sim$, where $\sim$ is the relation of  which identifies the open subsets $\ell_{\mathring{S}}^+$ and $\ell_{\mathring{S}}^-$ as

  \begin{equation}
    \label{eq:8}
    \R^+ \cdot  v_p  \sim \R^+\cdot w_p \Longleftrightarrow \R^+\cdot  v_p  = - \R^+\cdot  w_p, \qquad \forall \; \R^+ \cdot v_p \in \ell_{\mathring{S}}^+ \qquad \forall \; \R^+\cdot  w_p\in \ell_{\mathring{S}}^-.
  \end{equation}

\end{enumerate}

 
\end{theorem}

\begin{proof} All the assertions but \ref{item2thmMiguel} come from Theorem \ref{thm:main2:4}, Proposition \ref{p_Haus} and Lemma \ref{lemma:hypersurface}. 

 For \ref{item2thmMiguel}, it is clear that the unique cone geodesics which lie in both  $I_{\mathring{M}}^+(S)$ and $I_{\mathring{M}}^-(S)$
 are those crossing $\mathring{S}$ and, thus,  there is a natural bijection between  $\Ng$ and $(\ell^+_S \cup \ell^-_S)/\sim$.
Using the part \ref{item1thmMiguel}, $\Ng$ is topologically the quotient in the  (disjoint) union of $\ell^+$ and $\ell^-$ obtained by gluing each point and its image by the natural map $\ell_{\mathring{S}}^+ \rightarrow \ell_{\mathring{S}}^-$ consistent with \eqref{eq:8}.

The differentiable structure is inherited by the quotient because any point $\R^+\cdot v_p\in (\ell_{\mathring{S}}^+ \equiv \ell_{\mathring{S}}^-)$ which lie in the boundary of the glued points (i.e. with $p\in \partial S$), can be approximated only by points either in $\ell_{\partial S}^+$ or in 
$ \ell_{\partial S}^- $ (but not from both\footnote{Notice this is the key for Hausdorffness. Otherwise, two  non-Hausdorff related points would appear in the frontier of the  glued regions, as these regions were open (see Figure \ref{fig:5}).}), thus inheriting the differentiable structure of the approximating part. \end{proof}

\begin{figure}
  \centering
  \begin{tikzpicture}
    \path[black, line width=0.03cm, in=160, out=-90] (-3,4) edge (-1,0);
    \path[black, line width=0.03cm, in=20, out=-90] (3,4) edge (1,0);
    \path[black, line width=0.03cm, in=200, out=-20, dashed] (-1,0.25) edge (1,0.25);
    \path[black, line width=0.03cm, in=160, out=20, dashed] (-1,0.25) edge (1,0.25);
    \path[black, line width=0.03cm, in=200, out=-20] (-1,0) edge (1,0);
    \draw[dotted] (0,2) circle (2.7cm and 0.5cm);
\begin{scope}
    \clip (-3,2) rectangle (3,0);
        \draw (0,2) circle (2.7cm and 0.5cm);
\end{scope}
    \path[black, line width=0.03cm, in=200, out=90] (-3,-4) edge (-1.3,0.15);
    \path[black, line width=0.03cm, in=-20, out=90] (3,-4) edge (1.3,0.15);
    \draw[dotted] (0,-2) circle (2.9cm and 0.5cm);
\begin{scope}
    \clip (-3,-2) rectangle (3,-3);
        \draw (0,-2) circle (2.9cm and 0.5cm);
\end{scope}

    \draw [dotted, line width=0.4mm, color=black] (-0.85,0.2) -- (-1.5,-0.3);
        \draw [line width=0.4mm, color=black, -{Latex[length=3mm]}] (-1.08,0.03) -- (-1.7,-0.45);
  \end{tikzpicture}
  \caption{\label{fig:5} The dotted circle represents the boundary of two  identified disks  in the upper and lower surfaces. The depicted vector is tangent to both surfaces, so that it can be approximated by vectors tangent to both   the upper and  the lower surface. Along the circle, 
  the glued surface is either non-Hausdorff (if the gluing  only affects the open disks but excludes the corresponding circles) or not a manifold (if the two circles are also glued).}
\end{figure}
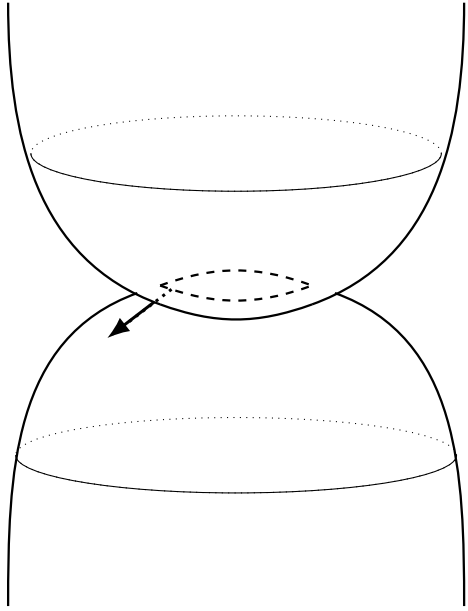

 The previous result allows us to obtain a precise description of the space $\Ng$, based on the fact that every geodesic must either intersect the boundary $\partial M$ or the Cauchy hypersurface $S$. Naturally, the inclusion of the corresponding sets $\ell^{\pm}_{\mathring{S}}$ enables us to capture those geodesics that do not reach the boundary $\partial M$.
 However, if one focuses on those that do intersect the boundary, we can achieve directly the following simpler description that does not require Cauchy hypersurfaces.

\begin{corollary} \label{cor:partialM}
 The open set $U^+$ (resp. $U^-$) $\subset \Ng$ given by all the $\lc{}$-geodesics with a future (resp. past) endpoint at $\partial M$  is diffeomorphic to $\ell^{-}_{\partial M}$ (resp. $\ell^{+}_{\partial M}$).
    
     In particular, if every cone geodesic has a future (resp. past) endpoint at $\partial M$  then $\Ng$ is diffeomorphic to  $\ell^{-}_{\partial M}$ (resp. $\ell^{+}_{\partial M}$), that is, to $\R\times T(\partial S)$. 
\end{corollary}
\begin{proof}
    The first assertion is straightforward and use Proposition   \ref{p_ell} for the last one.
\end{proof}

The simplified structure in the  corollary also appears in the following examples.

\begin{exe}\label{ejemplo3apart} {\em 
(1) Consider a product $\R\times S$ as in Convention \ref{convention} with  a  compact Cauchy $S$. Using a cone triple, the  cone structure can be written as $\lc{}=L^{-1}(0)$ with  
\begin{equation}\label{eq_example}
L=dt^2-F_t^2 .
\end{equation}
where $F_t$ is a $t$-depending Finsler metric ($F$ can be regarded as a Finsler metric on  $\mathrm{Ker}\,dt$, as in  the proof of Theorem \ref{t_4.4}).  Assume that infinitesimal lightconvexity holds and, additionally, the existence of a metric $F_\infty$ bounding all $F_t$ to the future,  that is, for some $t_0>0$,  $F_t\leq F_\infty$ for all  $t>t_0$. Then, the cones are  wider than those for $dt^2-F_\infty^2$ for $t>t_0$ and, necessarily, all the lightlike geodesics in $\R\times \mathring{S}$ will intersect $\partial M$ towards the future. Thus, Corollary \ref{cor:partialM} applies and $\Ng$ is identifiable to $\ell^-_{\partial M}$, thus, it is diffeomorphic to $\R\times T(\partial S)$. 

(2) Consider  metrics as in \eqref{eq_example} on $ I \times\withboundaryS$  with $\withboundaryS$ compact and  $I\subset \R$ is an interval (see 
Figure~\ref{fig:4}   for $n=1$ and in Figure~\ref{fig:6} for $n=2$),   
which are necessarilly globally hyperbolic with smooth timelike boundary. In the case $F_t$ independent of $t$, infinitesimal lightconvexity is equivalent to the usual convexity of the boundary. 
The simplest case when all the $F_t$ are just the usual Euclidean metric on a closed ball of radius $R>0$ in $\R^n$ may be illustrative,  as the representation of the space of lightlike geodesics may vary depending on $R$ and the length of $I$. 
}
\begin{figure}
  \centering
    \begin{tikzpicture}[scale=0.8]
      \begin{scope}[shift = {(-11,0)}]
      \draw [line width = 0.4mm] (-4,-5) -- node[below right] {} (-4,5);
      \draw [line width = 0.4mm] (4,-5) -- (4,5) node[above left] {\Large$[-1,1]\times I$};
      \draw [line width = 0.4mm] (-4,0) -- node[below right] {} (4,0);
      \node at (5.3,2) {$J^+(\withboundaryS)\cap \partial M$};
      \node at (-5.3,2) {$J^+(\withboundaryS)\cap \partial M$};
      \node at (5.3,-2) {$J^-(\withboundaryS)\cap \partial M$};
      \node at (-5.3,-2) {$J^-(\withboundaryS)\cap \partial M$};
      \node at (-4.3,0) {\Large$\withboundaryS$};
      \filldraw[fill=black] (-4,3) circle (0.65mm)node[below right] {$q$};
      \draw [line width = 0.2mm, dashed, color=gray] (-4,3) -- (-3.5,4);
      \filldraw[fill=white, draw=blue] (-3.5,4) circle (0.65mm) node[right]{$[q]^+$};
      \path [lightgray, out=65, in=-140,-{Latex[length=3mm]}] (-4,3) edge node[black, above right, pos=0.97] {{$\gamma_q$}} (-2.7,5);
      \filldraw[fill=black] (4,-3) circle (0.65mm)node[below right] {$r$};
      \draw [line width = 0.2mm, dashed, color=gray] (3.5,-4) -- (4,-3);
      \filldraw[fill=white, draw=red] (3.5,-4) circle (0.65mm) node[left]{$[r]^-$};
      \path [lightgray, out=55, in=-115,-{Latex[length=3mm]}] (2.7,-5) edge node[black, left, pos=0.1] {{$\gamma_r$}} (4,-3);
        \filldraw[fill=black] (-1,0) circle (0.65mm)node[below left] {$p$};    
        \draw [line width = 0.2mm, dashed, color=gray] (-1.5,-1) -- (-0.5,1);
        \draw [line width = 0.2mm, dashed, color=gray] (-1.5,1) -- (-0.5,-1);
        \filldraw[fill=white, draw=blue] (-0.5,1) circle (0.65mm) node[right]{$[p]_1^+$};
        \filldraw[fill=white, draw=red] (-1.5,-1) circle (0.65mm) node[left]{$[p]_1^-$};
        \filldraw[fill=white, draw=red] (-0.5,-1) circle (0.65mm) node[right]{$[p]_2^-$};
        \filldraw[fill=white,draw=blue] (-1.5,1) circle (0.65mm) node[left]{$[p]_2^+$};    
        \draw [lightgray,-{Latex[length=3mm]}] plot [smooth] coordinates{(-1.5,-2) (-1.25,-0.6) (-1,0) (-0.75,0.6) (-0.5,2)};
        \draw [lightgray,{Latex[length=3mm]}-] plot [smooth] coordinates{(-1.5,2) (-1.25,0.6) (-1,0) (-0.75,-0.6) (-0.5,-2)};
        \node at (-0.5,2.2) {$\gamma_{p,1}$};
        \node at (-1.5,2.2) {$\gamma_{p,2}$};
      \filldraw[fill=black] (4,0) circle (0.65mm)node[below right] {$s$};
      \draw [line width = 0.2mm, dashed, color=gray] (3.5,-1) -- (4,0);
      \filldraw[fill=white, draw=red] (3.5,-1) circle (0.65mm) node[left]{$[s]^-$};
      \path [lightgray, out=75, in=-115,-{Latex[length=3mm]}] (3.1,-2.2) edge node[black, left, pos=0.1] {{$\gamma^-_s$}} (4,0);
      \draw [line width = 0.2mm, dashed, color=gray] (3.5,1) -- (4,0);
      \filldraw[fill=white, draw=blue] (3.5,1) circle (0.65mm) node[left]{$[s]^+$};
      \path [lightgray, out=-75, in=115,{Latex[length=3mm]}-] (3.1,2.2) edge node[black, left, pos=0.1] {{$\gamma^+_s$}} (4,0);
    \end{scope}
    \draw [black] plot [smooth] coordinates{(3,5) (3,0.6) (2.3,0.2) (0.4,0.2) (-0.4,-0.2) (-2.3,-0.2) (-3,-0.6) (-3,-5)};
    \draw [black] plot [smooth] coordinates{(-3,5) (-3,0.6) (-2.3,0.2) (-0.4,0.2) (-0.2,0.1)};
    \draw [black] plot [smooth] coordinates{(0.2,-0.2) (0.4,-0.2) (2.3,-0.2) (3,-0.6) (3,-5)};
    \node at (2.5,5.4) {\Large $\Ng$};
    \filldraw[fill=white] (-3.05,3) circle (0.65mm)node[below right] {$[q]^{+}$};
    \filldraw[fill=white] (-1,0.2) circle (0.65mm)node[above right] {$[p]_{1}^{+}$};
    \filldraw[fill=white] (-1,-0.2) circle (0.65mm)node[below right] {$[p]_{2}^{+}$};
    \filldraw[fill=white] (2.9,-0.2) circle (0.65mm)node[below right] {$[s]^{-}$};
    \filldraw[fill=white] (2.9,0.2) circle (0.65mm)node[above right] {$[s]^{+}$};
    \filldraw[fill=white] (3.05,-3) circle (0.65mm)node[above left] {$[r]^{-}$};
    \end{tikzpicture}
         \caption{\label{fig:4} 
            The case of  $\withboundaryM= I\times [-1,1]$,  $I\subset \mathbb{R}$ an interval, $0\in I$   (a subset of $\Lo^2$). \\
           On the  left, $M$ and some points of $\Ng$ (depicted as unfilled points). Here all the depicted curves are lightlike geodesics intersecting either $\withboundaryS=  \{0\} \times   [-1,1]$ or $\partial M= I \times \left\{\pm 1 \right\}$. The dotted lines  represent $\pm \lc{}$   at each point.
    Observe that $[q]^+\in \ell^+_{\partial S }$ and $[r]^-\in \ell^-_{\partial  S  }$ represent  a  class of (unparametrized) lightlike geodesics starting from $q$ and ending at $r$ respectively. In this case, one  has only one class for each point in $\partial M\setminus \withboundaryS$.  Each $p\in S$ is crossed by two lightlike geodesics $\gamma_{p,i}$,  $i=1,2$. Then, $[p]^{+}_i\in \ell^+_{S}$ and $[p]^-_i\in \ell^-_S$  are $\sim$ related, according to \eqref{eq:8}. Hence, 
    we have two points in $\Ng$  (represented at the right figure by $[p]^{+}_{1}$ and $[p]^{+}_{ 2  }$).   Finally, for a point $s\in \partial S$ we also have two points $[s]^-$ and $[s]^+$ in $\Ng$, each one belonging to $\ell^+_{\partial M}\cap \ell^+_S$ and $\ell_{\partial M}^- \cap \ell^-_{S}$ respectively.\newline \indent On the  right, a  representation of the space $\Ng$ (which has two connected components), including the points at the left figure. }
\end{figure}
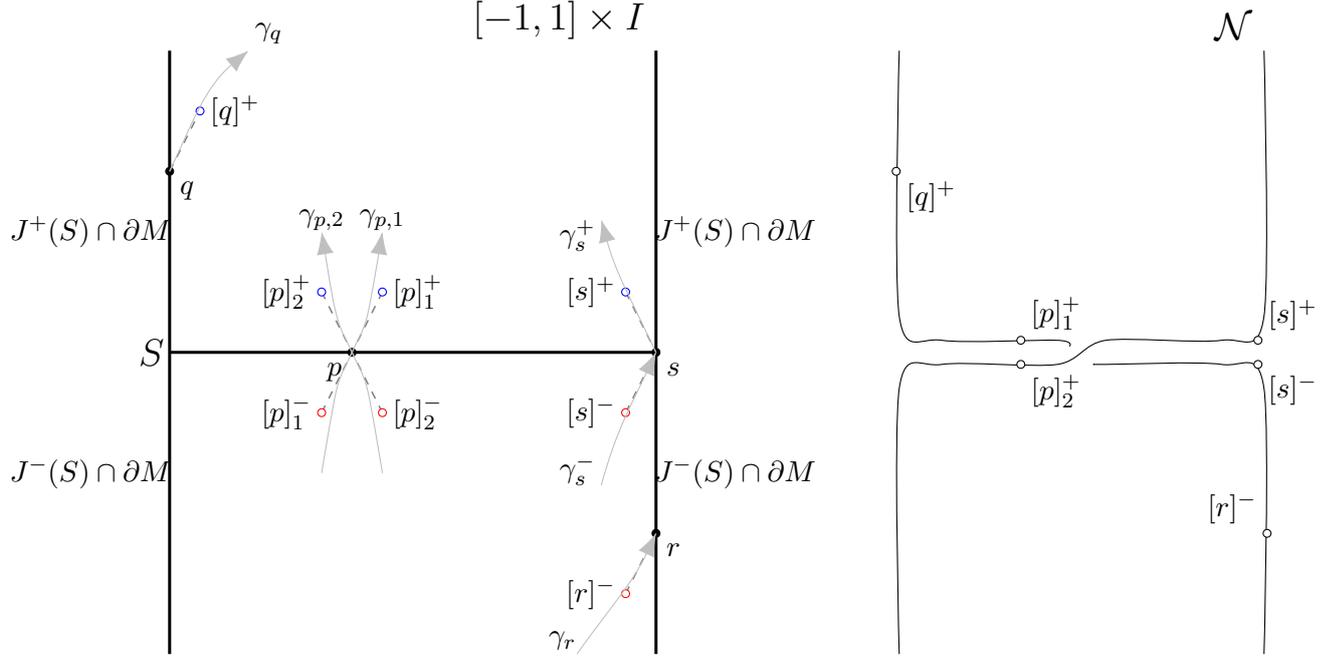

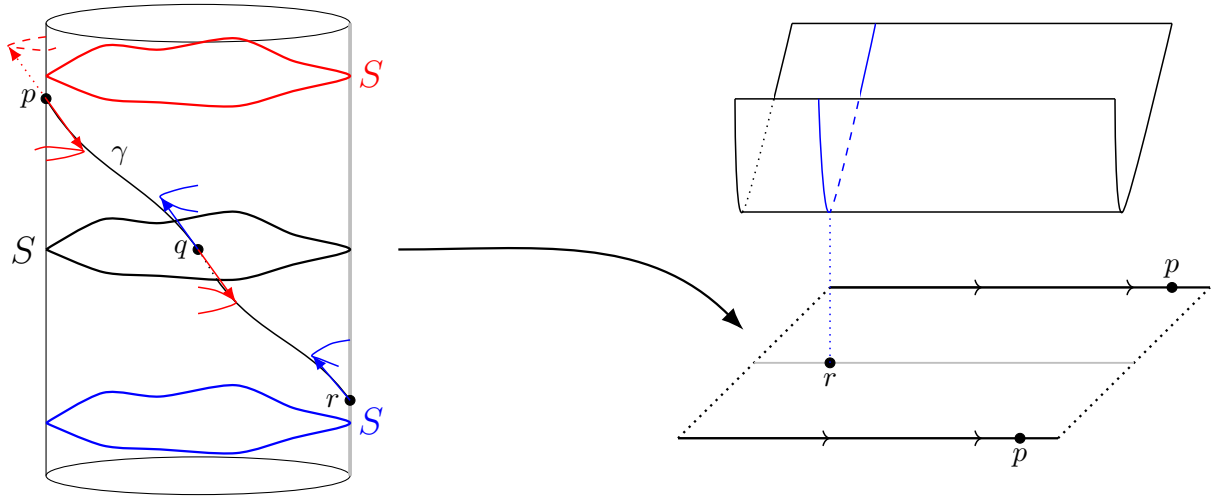
\begin{figure}
  \centering
\begin{minipage}{0.4\paperwidth}
  \begin{tikzpicture}[remember picture]
    \draw (0,3) circle (2cm and 0.25cm);
    \draw (0,-3) circle (2cm and 0.25cm);
    \draw (-2,-3) -- (-2,3);
    \draw[color = lightgray, line width= 0.5mm] (2,-3) -- (2,3);
    \draw [line width = 0.3mm] plot [smooth] coordinates{(-2,0) (-1.25,-0.3) (-0.5,-0.35) (0.5,-0.4) (1.25,-0.2) (2,0) (1.25, 0.2) (0.5, 0.5) (-0.5,0.35) (-1.25, 0.4) (-2,0)};
    \node[left] at (-2,0) {\Large $S$};
    \begin{scope}[shift={(0,-2.3)}]
            \draw [line width = 0.3mm, blue, mylabel= at 0.5 right with {\Large $S$}] plot [smooth] coordinates{(-2,0) (-1.25,-0.3) (-0.5,-0.35) (0.5,-0.4) (1.25,-0.2) (2,0) (1.25, 0.2) (0.5, 0.5) (-0.5,0.35) (-1.25, 0.4) (-2,0)};
    \end{scope}
        \begin{scope}[shift={(0,2.3)}]
            \draw [line width = 0.3mm, red, mylabel= at 0.5 right with {\Large $S$}] plot [smooth] coordinates{(-2,0) (-1.25,-0.3) (-0.5,-0.35) (0.5,-0.4) (1.25,-0.2) (2,0) (1.25, 0.2) (0.5, 0.5) (-0.5,0.35) (-1.25, 0.4) (-2,0)};
    \end{scope}
 \path [black, line width = 0.2mm, in = 125, out=-60] (-2,2) edge node[above] {$\gamma$} (0,0); 
 \path [dotted, black, line width = 0.2mm, in = 125, out=-55] (0,0) edge (0.25,-0.4);     
 \path [black, line width = 0.2mm, in = 125, out=-55] (0.25,-0.4) edge (2,-2);

\filldraw[fill=black] (-2,2) circle (0.65mm)node[left] {$p$};
\draw [dotted,line width = 0.2mm, color=red, -{Latex[length=2mm]}] (-2,2) -- (-2.5,2.7);
\draw [line width = 0.2mm, color=red, -{Latex[length=2mm]}] (-2,2) -- (-1.5,1.3);

\draw [dashed, color=red, line width = 0.2mm] plot [smooth] coordinates{(-2,2.82) (-2.25, 2.78) (-2.5,2.7) (-2.25, 2.66) (-2,2.64) (-1.85, 2.68)};

\begin{scope}[rotate around={180:(-2,2)}]
\draw [color=red, line width = 0.2mm] plot [smooth] coordinates{(-2,2.82) (-2.25, 2.78) (-2.5,2.7) (-2.25, 2.66) (-2,2.64) (-1.85, 2.68)};    
\end{scope}

\filldraw[fill=black] (0,0) circle (0.65mm)node[left] {$q$};
\draw [line width = 0.2mm, color=blue, -{Latex[length=2mm]}] (0,0) -- (-0.5,0.7);
\draw [line width = 0.2mm, color=red, -{Latex[length=2mm]}] (0,0) -- (0.5,-0.7);

\draw [color=blue, line width = 0.2mm] plot [smooth] coordinates{(0,0.5) (-0.25, 0.55)(-0.5,0.7) (-0.25, 0.8) (0,0.85)};

\begin{scope}[rotate around={180:(0,0)}]
    \draw [color=red, line width = 0.2mm] plot [smooth] coordinates{(0,0.5) (-0.25, 0.55)(-0.5,0.7) (-0.25, 0.8) (0,0.85) };
\end{scope}


\filldraw[fill=black] (2,-2) circle (0.65mm)node[left] {$r$};
\draw [line width = 0.2mm, color=blue, -{Latex[length=2mm]}] (2,-2) -- (1.5,-1.4);

\draw [color=blue, line width = 0.2mm] plot [smooth] coordinates{(2,-1.2) (1.75, -1.25) (1.5,-1.4) (1.6, -1.45) (1.85,-1.55)};
\node[name=salida] at (2.5,0) {};
\end{tikzpicture}
\end{minipage}
\begin{minipage}{0.35\paperwidth}
  \begin{tikzpicture}[remember picture]
    \draw[line width=0.3mm] (-2,-2) -- (3,-2) ;
    \draw[line width=0.2mm,->] (-2,-2) -- (0,-2);
        \draw[line width=0.2mm,->] (-2,-2) -- (2,-2) ;
        \draw[line width=0.3mm] (0,0) -- (5,0) ;
           \draw[line width=0.2mm, ->] (0,0) -- (2,0) ;
           \draw[line width=0.2mm, ->] (0,0) -- (4,0) ;
           \draw[dotted, line width = 0.3mm] (-2,-2) -- (0,0);
    \draw[dotted, line width = 0.3mm] (3,-2) -- (5,0);
    \draw[line width=0.25mm, color=lightgray] (-1,-1) -- (4,-1);
    \filldraw[fill=black] (0.,-1) circle (0.65mm)node[below] {$r$};
    \path[dotted, line width = 0.2mm, blue] (0,-1) edge (0,1);
      \filldraw[fill=black] (2.5,-2) circle (0.65mm)node[below] {$p$};
            \filldraw[fill=black] (4.5,0) circle (0.65mm)node[above] {$p$};
    \draw[line width = 0.2mm] (-0.5,3.5) -- (4.5,3.5);
    \draw[line width = 0.2mm] (-1.25,2.5) -- (3.75,2.5);
    \draw[line width = 0.2mm] (-1.15,1.) -- (3.85,1);
    \draw [line width = 0.2mm] plot [smooth] coordinates{(3.75,2.5) (3.85,1) (4.5,3.5)};

\begin{scope}[xshift=-0.15cm]
\clip (0,2.5) rectangle (0.15,1);
    \draw [line width = 0.2mm, color = blue] plot [smooth] coordinates{(0,2.5) (0.15,1) (0.75,3.5)};
\end{scope}    
\begin{scope}[xshift=-0.15cm]
\clip (0.15,1) rectangle (0.55,2.5);
    \draw [dashed,line width = 0.2mm, color = blue] plot [smooth] coordinates{(0,2.5) (0.15,1) (0.75,3.5)};
\end{scope}    
\begin{scope}[xshift=-0.15cm]
\clip (0.55,2.5) rectangle (0.75,3.5);
    \draw [line width = 0.2mm, color = blue] plot [smooth] coordinates{(0,2.5) (0.15,1) (0.75,3.5)};
\end{scope}    

    \begin{scope}
      \clip (-1.15,1) rectangle (-0.1,2.5);
    \draw [line width = 0.2mm, dotted] plot [smooth] coordinates{(-1.25,2.5) (-1.15,1) (-0.5,3.5)};      
    \end{scope}

    \begin{scope}
\clip (-2,-1.15,1) rectangle (-1.15,2.5);
      \draw [line width = 0.2mm] plot [smooth] coordinates{(-1.25,2.5) (-1.15,1) (-0.5,3.5)};
    \end{scope}

    \begin{scope}
\clip (-0.73,1) rectangle (-0.5,3.5);
      \draw [line width = 0.2mm] plot [smooth] coordinates{(-1.25,2.5) (-1.15,1) (-0.5,3.5)};
    \end{scope}    
    \node[name=llegada] at (-1,-0.7) {};
  \end{tikzpicture}
\end{minipage}
\begin{tikzpicture}[remember picture, overlay]
\path[line width = 0.3mm, out=0, in=135, -{Latex[length=3mm]}](salida) edge (llegada);    
\end{tikzpicture}
\caption{\label{fig:6}  In the left-hand figure,  a single geodesic $\gamma$ determines a point in $\ell^+_{\partial M}$  (which can be also seen as a point in $\ell^-_{\partial S}$ for red Cauchy $S$), one in $\ell^{-}_{\partial M}$ (which can be also seen as a point in $\ell^+_{\partial S}$ for blue Cauchy $S$)  
and one in $\ell^{\pm}_S/\sim$ for black Cauchy $S$. Accordingly,   these points are represented by red and blue vectors (oriented directions), respectively. Also depicted are the subsets of $\Ng$ corresponding to the pre-geodesics that end at $p$ (shown as a red arc), start at  $r$ (a blue arc), and pass through $q$ close to $\gamma$ (represented by two arcs blue and red, whose points in opposite directions with respect $q$ should be identified).\\
On the right-hand side, the horizontal rectangle  represents the cylindrical boundary of the figure on the left cutting it by de vertical line at $p$. Above  it, a pictorial representation of the subset of $\Ng$ projecting onto the vertical gray line through $r$. The blue arc from the left figure is also included  
in this representation. This arc has been rotated, allowing it to be visualized as a fibered space over the gray line (but all the blue arc projects on $r$).}
\end{figure}
\end{exe}

\begin{remark}[Cone geodesics at totally lightgeodesic $\partial M$] \label{r_totallylightconvex}
    {\em   In our analysis, lightlike geodesics $\gamma$ or directions tangent to  $\partial M$  are always excluded, as they cannot enter in $\mathring{M}$ and, thus, depending on the sign of $II(\gamma',\gamma')$ either remain at $\partial M$ or would be pushed out $M$. Thus, it is not natural to consider the space of cone geodesics of the whole $M$ except when $\partial M$ {\em is totally lightgeodesic}, i.e., $II(w,w)=0$ for all lightlike $w$ tangent to $\partial M$. 
    
     Recall that (each connected component of) $\partial M$ is intrinsically globally hyperbolic \cite[Theorem 3.8]{RMI_AFS} and, thus, its lightspace is a smooth manifold $\Ng_{\partial M}$.  In the case that $\partial M$ is totally lightgeodesic, $\Ng_{\partial M}$  matches  $\Ng$  becoming a manifold with boundary.
    } 
\end{remark}

\subsection{Applications to asymptotically anti-de Sitter spacetimes}\label{s_4.3}

Next, our previous results will be applied to a significant class  of spacetimes, the \emph{asymptotically Anti-de Sitter}   (asymptotically AdS for short) ones. Up to a conformal factor 
(irrelevant for our aims in causality), these spacetimes become globally hyperbolic with a timelike boundary, where the latter is  infinitesimally null convex. Moreover, $\ell_{\partial M}^{\pm}$ will be computed and interpreted.

\subsubsection{The case of AdS}

AdS spacetime is a complete simply connected Lorentzian manifold $(M^{n+1},g_{AdS})$ whose metric  has constant negative sectional curvature, which will be chosen equal to -1 without loss of generality (in particular, our results are preserved by homotheties). In natural coordinates up to $\{0\}$, 
\begin{equation}
\label{pend:eq:90}
g_{AdS}=- \left( 1+r^2 \right) dt^2 + \left( \frac{1}{1+r^2} \right)dr^2 + r^2 g_{\mathbb{S}^{n-1}},\;\;(t,r,\theta)\in M:= \mathbb{R}\times \mathbb{R}_+\times \mathbb{S}^{n-1},
\end{equation}
where $(\mathbb{S}^{n-1},g_{\mathbb{S}^{n-1}})$ denotes the $(n-1)$-dimensional sphere with its usual metric. 

\begin{lemma}\label{l_AdS}
      Each hypersurface $H_r:=\left\{ \mathbb{R} \right\} \times \left\{ r \right\}\times \mathbb{S}^{n-1}$ is   infinitesimally (strongly) lightconvex. Moreover, $H_r$ is the boundary of a globally hyperbolic spacetime-with-timelike-boundary.
\end{lemma}
\begin{proof}
    Any lightlike vector tangent to $H_{r}$ is proportional to $X^{\pm}_v=\frac{\partial}{\partial t}\pm \frac{\sqrt{1+r^2}}{r}v$, where $v$ is a (locally extended) unit vector  on $(\mathbb{S}^{n-1},g_{\mathbb{S}^{n-1}})$. Recalling \eqref{eq:1} and Definition \ref{def:main2:12},  let us  compute $\second{\eta}$ on $X^{\pm}_v$ choosing the rigging vector field $\eta=-\left(\sqrt{1+r^2} \right)\frac{\partial}{\partial r}$. As $\eta$ is  spacelike, unit and  normal to $H_r$ we have (using Koszul's formula for the second line): 
\begin{equation}
\label{pend:eq:200}
\begin{split}
\second{\eta}(X^{\pm}_v,X^{\pm}_v)=&{g}_{AdS}(\nabla_{X^{\pm}_v}X^{\pm}_v,\eta)  \\ = & - \frac{\sqrt{1+r^2} }{2} \left(  2 X^{\pm}_v {g}_{AdS}(\frac{\partial}{\partial r}, X^{\pm}_v)- \frac{\partial}{\partial r} {g}_{AdS}(X^{\pm}_{v}, X^{\pm}_{v}) + 2{g}_{AdS}(\mathcal{L}_{\frac{\partial}{\partial r}}X^{\pm}_v,X^{\pm}_v)   \right)  \\ = & -\left(\sqrt{1+r^2} \right){g}(\mathcal{L}_{\frac{\partial}{\partial r}}X^{\pm}_v,X^{\pm}_v) = -\left(\sqrt{1+r^2} \right){g}(\pm \partial_r \left(\frac{\sqrt{1+r^2}}{r}  \right)v,X^{\pm}_v)   \\
  =& -\left(\sqrt{1+r^2} \right)r^2 \partial_r \left(\frac{\sqrt{1+r^2}}{r}  \right) \frac{\sqrt{1+r^2}}{r}  = -\frac{\sqrt{1+r^2}}{2}r^2 \partial_r \left(\frac{1+r^2}{r^2}  \right) = \frac{\sqrt{1+r^2}}{r}
\end{split}
\end{equation}
which is positive,  as required. The last assertion is straightforward because, for any $p,q$ in the region $r\leq r_0$, the causal diamond $J^+(p)\cap J^-(q)$ lies in the compact region $t(p)\leq t \leq t(q)$.
\end{proof}
\begin{remark}\label{r_AdSlighttotallygeodesic0} 
{\em The positiveness of $^\eta II$ suggests infinitesimal lightconvexity of AdS spacetime at its asymptotic boundary $\partial$AdS and, moreover,  the  limit $r\rightarrow\infty$ in \eqref{pend:eq:200} is $1$. From Proposition~\ref{p_invar_conf_lightconvex}, both lightconvexity and strong lightconvexity are conformally 
invariant. As will be clear next,  the latter does not apply directly to $\partial$AdS, because the metric $g_{AdS}$ is not defined therein and the conformal factor to be chosen for the extension of the metric will go to 0 on $\partial$AdS. 
} \end{remark}
Next, a conformal representative of the class extensible along $\partial$AdS  will be  used. The conformal factor is $1/r^2$ for the region $r \geq 2$ and it is 1 when $r\leq 1$), namely, 
\begin{equation}
  \label{pend:eq:100}
  \overline{g}_{AdS}=\frac{1}{r^{2}}\,g_{AdS}= - \left( \frac{1+r^2}{r^2} \right) dt^2 + \left( \frac{1}{r^2(1+r^2)} \right)dr^2 + g_{\mathbb{S}^{n-1}}, \qquad r \geq 2,
\end{equation}
and change $r$ by the variable $dz=\frac{1}{r \sqrt{1+r^2}}dr$ with $z(2)=0$. Putting $z_*=\int_2^{\infty}\frac{1}{r \sqrt{1+r^2}}dr\;(<\infty)$, the expression \eqref{pend:eq:100} becomes 
\begin{equation}
\label{pend:eq:9}
\overline{g}_{AdS}=-f(z)dt^2 + dz^2 + g_{\mathbb{S}^{n-1}},\;\; (t,z,\theta)\in \mathbb{R}\times ( -\infty,  z_*)\times \mathbb{S}^{n-1},
\end{equation}
where $f(z(r))= \frac{1+r^2}{r^2}$. Notice that the origin $\{0\}$ in \eqref{pend:eq:90} is mapped into $z=-\infty$, but this will be irrelevant next and one can focus on $z>0$ $(r>2)$.  In particular, 
$$\lim_{z\rightarrow z_*}f(z)=\lim_{r\rightarrow \infty}\frac{1+r^2}{r^2}=1, \quad  \lim_{z\rightarrow z_*}\frac{df}{dz} \left(= \lim_{z\rightarrow z_*}\frac{dr}{dz} \; \cdot \;  \lim_{r\rightarrow \infty}\frac{d(f\circ z)}{dr}\right) =0, \quad \lim_{z\rightarrow z_*}\frac{d^2f}{dz^2}=2.$$
So,  
 this conformal AdS representative  can be  $C^2$-extended\footnote{Indeed, we only need a $C^{1,1}$ extension for our needs on convexity, even if smoother extensions are possible.}  to $z=z_*$. 
 \begin{lemma}\label{l_AdS2}
      $\withboundaryM=\mathbb{R}\times ( -\infty,z_*]\times \mathbb{S}^{n-1}$ (up to the origin) endowed with the metric \eqref{pend:eq:9} is a conformal extension of AdS as a spacetime-with-timelike-boundary  $\partial M=\{z=z_*\}$. Moreover, $\partial M$ is totally geodesic. 
 \end{lemma}
\begin{proof}
    The first assertion has already been proved. For the last one, the fact 
     $(df/dz)(z_*)=0$ easily implies the vanishing of $^\eta II$ (for example,  $-f(z)dt^2+dz^2$ can be seen as a warped 
    product and O'Neill's formulas \cite[Prop. 7.35(2)]{oneill} apply).
\end{proof}
 
 \begin{proposition}\label{p_AdS} The conformal class of  AdS spacetime is globally hyperbolic with timelike boundary, and this boundary is lightconvex . 
  \end{proposition}

\begin{proof} 
It is enough to check the stated properties for the  conformal representative $\overline{g}_{AdS}$ in \eqref{pend:eq:9}. Global hyperbolicity follows as in  Lemma \ref{l_AdS}, because the causal diamonds 
$J(p,q)$ lie in the  set  $t(p)\leq t \leq t(q)$, which is compact when the boundary $ z= z^*$ is included.  The properties on the boundary 
are proven in  Lemma \ref{l_AdS2}.
\end{proof}

\begin{remark}\label{r_AdSlighttotallygeodesic} {\em 
    Following the discussion in Remark \ref{r_AdSlighttotallygeodesic0}, we have proven that $\partial$AdS is {\em not} strongly lightconvex but only lightconvex for the above conformal extension.
Notice that the hypersurfaces  $H_r$ at constant r, whose $g_{AdS}$-second fundamental form was lowerly bounded by $1$ in Lemma \ref{l_AdS}, are equal to the hypersurfaces at constant $z$ whose $\bar g_{AdS}$-one goes to 0, making  it  vanish on all the  lightlike vectors, that is, $\partial$AdS becomes totally lightgeodesic. As 
Proposition~\ref{p_invar_conf_lightconvex} does apply to any conformal extension of the metric on $\partial$AdS, we conclude: {\em $\partial$AdS is intrinsically totally lightgeodesic for any $C^{1,1}$ extension of its cone structure}. In particular, the observations on the lightlike geodesics of the boundary in
Remark \ref{r_totallylightconvex} applies. 
}\end{remark}

The full description of $\Ng$ is easy from Corollary \ref{cor:partialM}.

\begin{corollary}
 Any inextensible lightlike geodesic in the AdS spacetime intersects the boundary $\partial M$ both in the future and in the past (in finite time $t$). 
 
Thus, its space of cone geodesics $\Ng$ is diffeomorphic to $\R\times T\mathbb{S}^{n-1}$.

\end{corollary}
\begin{proof}
To prove the first assertion suffices  by Corollary \ref{cor:partialM}. 
Using conformal invariance, we will consider the metric \( g_{AdS} \). 
The required property 
is well known from its Penrose diagram (see for example \cite{HawkingLargeScaleStructure1975}), but  a  computation is included for completeness. As  AdS spacetime  is homogeneous, it is enough to consider a single  lightlike  geodesic, 
\(\gamma(s) = (t(s), r(s), \theta(s))  \), thorough the origin and, thus, with constant $\theta$ 
(i.e.  zero Killing constants $
g(\partial_\theta, \dot{\gamma}(s))$).  
Using that $\gamma$ is lightlike in \eqref{pend:eq:90},  
\begin{equation} \label{AdSGeodesics}
  \begin{split}
    0 = g(\dot{\gamma}(s), \dot{\gamma}(s)) &= -(1 + r(s)^2)\dot{t}(s)^2 + \left(\frac{1}{1 + r(s)^2}\right)\dot{r}(s)^2. 
  \end{split}
\end{equation}
Thus,   $\dot r$ cannot vanish and, integrating $dt/dr= \pm 1/(1+r^2)$ one has that $r$ must diverge (otherwise, $\gamma$ would have an endpoint and, choosing a convex neighborhood around it, $\gamma$ would be extensible as a geodesic beyond). Then, the increase of the $t$-coordinate is finite when $r\rightarrow \infty$, and $\gamma$ arrives at an endpoint at the boundary ($r=\infty$), identifiable to $z=z_*$ in \eqref{pend:eq:9}. 
\end{proof}




\subsubsection{Asymptotically AdS spacetimes}


Asymptotically AdS spacetimes should preserve its boundary and infinitesimal lightconvexity. Next, this statement will be developed accurately. 

\begin{definition}
\label{def:Cuestionespendientes:1}
A spacetime $(M,g)$ is \emph{asymptotically anti-de Sitter} if $g=g_{AdS}+h$, where $h$ is a $(0,2)$-tensor satisfying, in the coordinates $h_{ij}(t,r,\theta)$ of \eqref{pend:eq:90} (with $\theta$  spherical ones in $\mathbb{S}^{n-1}$),

\begin{align}
\label{pend:eq:11}
\lim_{\bar r\rightarrow \infty} h_{ij}(t,\bar  r,\theta) = \lim_{\bar r\rightarrow \infty} \partial_rh_{ij}(t,\bar  r,\theta)=0, \qquad \qquad \qquad \qquad \hbox{for all} \quad  t , \theta,  
\\  
\label{pend:eq:11-1}
 |\partial_\theta h_{ij}(t,r,\theta)|, \;  |\partial_t h_{ij}(t,r,\theta)| < A_m \qquad \quad  \hbox{for all} \;  |t|\leq m, \; \hbox{and} \, \; r>0, \theta  . 
 \end{align}
where the bounds $A_m>0$ are possibly divergent.
\end{definition}

 This notion of asymptotically AdS is weaker than the standard definitions commonly found in the literature \cite{Henneaux1985, Fefferman-Graham-1985}
and could be even sharpened, anyway, it will be sufficient for our purposes. From the definition, equation~\eqref{pend:eq:11} guarantees that the metric  $g$ asymptotically approaches the AdS metric  $g_{\text{AdS}}$ for large values of $r$. By performing the same conformal factor and coordinate transformation as in the AdS case (see around \eqref{pend:eq:9}), we obtain a  metric $\overline{g}$  conformal to  $g$ in the coordinates $(t, z, \theta) \in \mathbb{R} \times (-\infty, z_*] \times \mathbb{S}^{n-1}$, namely:

\begin{equation}
\label{pend:eq:12}
\overline{g} = \overline{g}_{AdS} + r(z)^{-2}h.
\end{equation}
Equations \eqref{pend:eq:11} and \eqref{pend:eq:11-1} allow us to deduce both,

\begin{equation}
  \label{pend:eq:20}
  \begin{split}
    \lim_{z\rightarrow z_*} r(z)^{-2}h_{ij}=\lim_{z\rightarrow z_*}\partial_z \left( r(z)^{-2}h_{ij} \right)=0,\\
    \lim_{z\rightarrow z_*} r(z)^{-2} \partial_\theta h_{ij} = \lim_{z\rightarrow z_*} r(z)^2 \partial_t h_{ij} =0.   
  \end{split}
\end{equation}
Summing up:
\begin{proposition}
The metric $\overline{g}$ can be $C^{1,1}$ extended to the boundary $z=z^*$ and both $\overline{g}$ and its first derivatives are equal to those for $\overline{g}_{AdS}$ therein. 
\end{proposition}

\begin{remark} \label{r_asin_AdS}{\em 
     As a consequence, the boundary  becomes totally geodesic, the extended metric yields a globally hyperbolic spacetime-with-timelike-boundary and all the consequences of these facts   (derived for the case of AdS) hold, as summarized in Corollary \ref{c1}. 
     
     By Proposition \ref{p_ell},   $\ell^+_{TM}$ is diffeomorphic to $\R\times T\mathbb{S}^{n-1}$, which becomes then an open subset of $\Ng$. However, $\Ng$ may contain more points in the properly asymptotic case. For example, a modification of the metric in a solid cylinder $r<R_0$ (maintaining $t$ as a temporal function) would preserve the asymptotic AdS character but might make a lightlike geodesic remain in the cylinder.} 
\end{remark}

\section{A prospective step:  causally simple spacetimes 
with $T_2$-lightspace 
} \label{s_last}

As we have seen, in the class of globally hyperbolic cone structures,   causal simplicity is  equivalent to the  Hausdorffness of $\Ng$ for the interior $\mathring{M}$. However, this happens because of the equivalence of these properties with the lightconvexity of the boundary $\partial M$, but not in general. Indeed, for causal $\lc{}$, causal simplicity is equivalent to the closedness of the causal relation $J$ (or the horismotic one $E$)  as a subset of $M\times M$ (see \cite[Proposition 3.68 (iii)]{MS}, but  
Proposition \ref{p_Rui} shows that the Hausdorffness of $\Ng$ 
depends on the closedness of the relation which defines $\Ng$ as a subset of $T\lc{} \times T\lc{}$. 
This suggest to introduce the following step in the causal ladder of spacetimes/cone structures, as an intermediate one between causally simple and globally hyperbolic.

\begin{definition} A Finsler spacetime or cone structure 
is {\em  causally simple with $T_2$-lightspace} if it is causally simple and its space of lightlike (cone) geodesics $\Ng$ is 
    Hausdorff.
\end{definition}
The  interior of globally hyperbolic spacetimes-with-timelike-boundary and lightconvex $\partial M$ is a relevant class of this new step of the causal ladder and, next, we will stress the interest of this step with other results and examples.

\subsection{Some proper (non-globally hyperbolic) examples} 

First, we will consider a  complementary class in the new level and, then, 
 both classes will be exemplified  with the family of stationary spacetimes.

\subsubsection{Causally simple spacetimes with no lightlike cut points}

The following class permits to consider the case when the boundary is not regular (or makes no sense). Recall that a lightlike geodesic $\gamma$ is free of cut points if each two points on it  are horismotically related, that is, the lightlike geodesic is maximizing for the time separation of any compatible Lorentz-Finsler metric. 

\begin{proposition}\label{p_nocutlight} Let $\lc{}$ be a causally simple cone structure (without boundary). If its cone geodesics are free of cut points then $\Ng$ is Hausdorff.  
\end{proposition}
\begin{proof}
    Assume by contradiction that there exists a sequence of inextendible lightlike geodesics $\{\gamma_m\}$,  $m>3$, converging to other two geodesics $\gamma_1, \gamma_2$ and, thus, two sequences of horismotically related points $\gamma_m(s_m)=p_m \hookrightarrow q_m=\gamma_m(t_m)$ (with $s_m< t_m$) on each $\gamma_m$ and two points $p=\gamma_1(s_\infty)$, $q=\gamma_2(t_\infty)$
    such that the corresponding directions of the velocities converge, that is, 
    $$\gamma'_{m}(s_m)\rightarrow \gamma'_1(s_\infty), \qquad  
    \gamma'_{m}(t_m)\rightarrow \gamma'_2(t_\infty), \qquad s_\infty<t_\infty.$$ 
    Considering convex neighborhoods $U_p, U_q$ around $p$ and $q$ resp., we can also choose points $p_m'=\gamma_m(s'_m)$, $p'=\gamma_1(s'_\infty)$ in $U_p$ and $q_m'=\gamma_m(t'_m)$, $q'=\gamma_2(t'_\infty)$ in $U_q$ such that 
    $$p_m \hookrightarrow p'_m \hookrightarrow q'_m  \hookrightarrow q_m \quad \hbox{and} \quad \gamma'_{m}(s'_m)\rightarrow \gamma'_1(s'_\infty), \quad 
    \gamma'_{m}(t'_m)\rightarrow \gamma'_2(t'_\infty), \; \quad  s_\infty<s_\infty'<t_\infty'<t_\infty.  $$ 
    By causal simplicity, the relation $\leq $ is closed and
    $$p\hookrightarrow p' \leq q' \hookrightarrow q .$$
    
    Then, there exists a causal curve $\rho$ from $p'$ to $q'$. This curve must be a lightlike pregeodesic and the direction of its velocity at $p'$ and $q'$ must agree with $\gamma_1(s_\infty')$ and $\gamma_2(t'_\infty)$. Indeed, otherwise $p\ll q$ and, for large $m$, $\gamma'_m(s_m)\ll \gamma'_m(t_m) $ (as the relation $\ll$ is open), in contradiction with the maximizing character of each $\gamma_m$. But then $\gamma_1$, $\rho$ and $\gamma_2$ match in a single pregeodesic,  a contradiction again. 
\end{proof}
\begin{remark}\label{r_nocutlight}
    {\em (1) The Lorentzian product $\Lo\times \SSS ^n$ shows that, obviously, the converse to Proposition~\ref{p_nocutlight}  does not hold (even if $\lc{}$ is globally hyperbolic). 

    (2)  The Lorentzian product $\Lo\times S$ where $S$ is any convex open subset of a Cartan-Hadamard manifold (for example, the interior of a square in $\R^2$) lies in the assumptions of Proposition \ref{p_nocutlight}. However, it is not a (smooth) spacetime with timelike boundary.
    }
\end{remark}

\subsubsection{Standard stationary spacetimes} \label{s_stationary}
Illustrative examples of globally hyperbolic spacetimes-with-timelike-boundary can be found among standard stationary ones, 
as their  causal properties  naturally involve  Finslerian elements (which become Riemannian in the case of Lorentzian products). 
Specifically, consider a standard stationary spacetime
\begin{equation}
\label{eq:19}
M = \mathbb{R}\times  S\quad \text{with} \quad g = -dt^2  + 2\omega dt + g_S,
\end{equation}
where \((S, g_S)\) is a Riemannian manifold (possibly with boundary) and $\omega$ is (the pullback of) a 1-form on\footnote{The general definition of standard stationary spacetime includes the lapse $\Lambda$ on $S$ multiplying $dt^2$. However, it can be reabsorbed here in the other terms with the conformal change $g\mapsto g/\Lambda$.} $S$ . The Fermat metric on $S$ is the Finsler metric $$F=\omega  + \sqrt{g_S+\omega^2}.$$
Even though the cone structure come from quadrics, the description of the causal properties in terms of objects on $S$ involves Finsler elements, as extensively studied in \cite{Caponio_2011, FHS_Memoirs}.  Indeed, 
$F$ reduces to $g_S$ in the static case ($\omega=0$), otherwise, $F$ is a non-reversible Finsler metric and its associated distance $d_F$ is non-symmetric. In particular, one has to distinguish between forward and backwards balls (depending whether the distance is taken from the center of the ball to the point or the the other way round).
Following \cite[Proposition 4.1]{Caponio_2011}, a curve \(\gamma(s) = (t(s),\sigma(s))\) is a (future-directed) lightlike (pre-)geodesic if and only if \(\sigma\) is a geodesic for $(S,F)$ and \(t'(s) = F(\sigma'(s))\). This implies that, when $s$ increases, the cutpoints of $\sigma $ and $\gamma$ agree and, moreover, the following results also hold:
\begin{itemize}
    \item \cite[Proposition 4.3]{Caponio_2011}. In the case without boundary: 
    
    $ (M, g)$ is causally simple if and only if $(S, F)$ is convex, in the sense that each $p,q\in S$ can be joined by a  geodesic from $p$ to $q$ minimizing the distance $d_F$.  

$ (M, g)$ is globally hyperbolic if and only if $(S, F)$ is complete, in the (weakest) sense that the intersections between its closed forward and backward balls 
are compact.
    
    \item \cite[Theorem 1.3]{Bartolo_2010}. Assume that $S$ has a  smooth  boundary $\partial S$ and $(S,F)$ is complete (in the sense above). $(\mathring{S}, F)$ is convex if and only if $\partial S$ is convex.   
        \end{itemize}

Then, using Proposition \ref{p_nocutlight}:

\begin{proposition}\label{p_standardstationaryNOCUT}
A standard stationary spacetime (as above) with no boundary is causally simple with $T_2$-lightspace if $(S,F)$ is convex and its geodesics has no cut points. 
\end{proposition}
        
The cited ideas in \cite{Caponio_2011} above extend directly to the case  with timelike boundary yielding:

\begin{proposition}
    A standard stationary spacetime with timelike boundary 
$(M, g)$ 
is globally hyperbolic  if and only if $(S, F)$ is complete, in the  sense that the intersections between its closed forward and backward balls (in the whole  $S=\mathring{S} \cup \partial S) $ are compact.
In this case, they are equivalent: 

(a)  $ (\mathring{M}, g)$ is causally simple, and 

(b)   $\partial S$ is convex with respect to the Fermat metric $F$.

\end{proposition}
As a consequence of Theorem \ref{thm:main2:4}, one can add a third equivalence above: (c)  $\partial M$ is lightconvex for $g$. However, convexity is a local property and any point of $\partial M$ admits a  neighborhood which is globally hyperbolic with timelike boundary. So, one has the following sharper consequence (that can also be obtained by algebraic computations at each point).

\begin{corollary}
    For a standard spacetime with timelike boundary, 
$\partial M$ is lightconvex for $g$ if and only if $\partial S$ is convex for the Fermat metric $F$.
\end{corollary}

\subsection{Some proper (causally simple) counterexamples}\label{s_counterexamples}

Next, we will give two illustrative examples of causally simple spacetimes with non $T_2$-lightspace.
The first one  refines  Hedicke and Suhr's \cite{Hedicke_2019} and stresses the importance of the smoothness of the timelike boundary (compare with Remark \ref{r_nocutlight} (2)).
The second one is inspired in  a Riemannian example  in \cite{BGS} about subtle properties on convexity.

Both counterexamples are products $\Lo\times S$. Thus, they reduce to construct a Riemannian manifold \((S, g_S)\) that is convex (but not complete) and contains a sequence of geodesics that naturally converges (in the sense explained at the end of  \S \ref{subsub_dif_est_n}) to two distinct geodesics. It is interesting to check the appareance of cut points and impossibility of a smooth timelike boundary, according to our previous results.


\subsubsection{A simple flat Lorentzian product}
\label{simpleflatlorentzian}
Let $(S,g_S)$ be a conic surface without vertex constructed as follows (see Figure \ref{fig:2}). Let $(r,\theta)$ be the polar coordinates in $\mathbb{R}^2$ and $\overline{S}:=\left\{ (r,\theta):r\in (0,1),\theta\in [\pi/4,-\pi/4]\right\}$. Define  $S=\overline{S}/\sim$, where $\sim$ is the relation of equivalence  gluing together the segments with $\theta=\pi/4,-\pi/4$ by means of  local rotations of $\pi/2$. 

From the construction, the cone $S$ is convex (notice that any two points can be joined by a minimizing geodesic) and, thus, the associated Lorentzian product $\Lo \times S$ is causally simple.  However, a glance at 
Figure~\ref{fig:2} shows that there are two minimizing geodesics   from each dotted point  to each point in the identified  arrow (by the rotational symmetry of the model, this property extends to any pair of points in opposed radial segments from the origin). Thus, these geodesics attain  a cut point therein and Proposition~\ref{p_nocutlight} will not be applicable. 

To check that  $\Ng$ is not Hausdorff, just consider in Cartesian coordinates $(x,y)$ the sequence of inextendible geodesics $\sigma_k:(a_k^-,a_k^{+})\rightarrow S$,  $\sigma_k(0)=(-1/k,0)$ and $\sigma'(0)=\partial_{y}$ (thus, $a_k^{\pm}\rightarrow \pm 1$). Then, the sequence $\left\{ \sigma_k \right\}_k$ converges to $\sigma_{\pm}$, where  $\sigma_-:(-1,0)\rightarrow S$, $\sigma_+:(0,1)\rightarrow S$ and $\sigma_{\pm}(s)=(0,s)$, as depicted in Figure \ref{fig:2}.
Thus, the lifted sequence of lightlike geodesics $\gamma_k:(-a_k,a_k)\rightarrow M=\R\times S$, $\gamma_k(s)=(s,-1/k,s)$ converge to both 
$ \gamma_+:(0,1)\rightarrow M, \gamma_{-}:(-1,0)\rightarrow M$, $\gamma_{\pm}(s)=(s,0,s)$, as required. 
\begin{remark}
{\em It is worth pointing out that any two points $p_{\pm}\in \gamma_{\pm}$ must be chronologically related. This comes from general arguments, as these $(p_-,p_+)$ lie in the boundary of the causal relation $J\subset M\times M$ (which is closed by causal simplicity) and they cannot be connected by a lightlike geodesic. Anyway, it  can be checked directly because of the mentioned property about existence of cut points for geodesics connecting radially opposed points. }   
\end{remark}




\begin{figure}
  \centering
  \begin{tikzpicture}[scale=1.3,every node/.style={transform shape}]
    \pic{carc=30:330:2cm} [line width = 1mm];
    \draw [line width = 1mm] (0,0) -- +(30:2.04cm);
    \draw [line width = 1mm] (0,0) -- +(-30:2.04cm);
    \draw [line width = 0.6mm, ->] (0,0) -- +(-30:1.7cm);
    \draw [line width = 0.6mm, ->] (0,0) -- +(-30:1.3cm);
    \draw [line width = 0.6mm, ->] (0,0) -- +(30:1.7cm);
    \draw [line width = 0.6mm, ->] (0,0) -- +(30:1.3cm);
    \foreach \x in {1,1.5,3,6}{
      \filldraw[fill=black] (-1/\x,0) circle (0.65mm);
      \draw [dashed] (-1/\x,-1.5) -- (-1/\x,1.5);
    }
    \draw [dashed,line width = 0.7mm] (0,-1.5) -- (0,1.5);
    \filldraw[fill=white, draw=black] (0,0) circle (0.85mm);
    \node at (-1.3,1) {\large$\sigma_k$};
    \node at (0.6,1.3) {\large$\sigma_{+}$};
    \node at (0.6,-1.3) {\large$\sigma_{-}$};    
  \end{tikzpicture}
  \caption{\label{fig:2} A pictorial representation of the Riemannian manifold $(S,g)$ described at the beginning of \S \ref{simpleflatlorentzian}. The identification of the arrowed radii yields a convex incomplete conic surface. The vertical segments $\sigma_k$ are geodesics converging to both vertical limits $\sigma_-$ and $\sigma_+$. The product $\Lo\times S$ is causally simple with no $T_2$ lightspace.}
\end{figure}
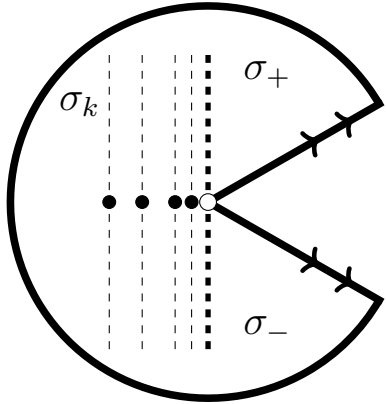

\subsubsection{Causal simplicity with $T_2$-lightspace non-preserved by a limit} \label{s522}

The example in \cite[Section 2.1]{BGS} shows a non-convex Riemannian   manifold $S$ without boundary exhausted by a  sequence of  compact domains $\{D_m\}, D_m \subset \mathring{D}_{m+1}$, each one with convex boundary $\partial D_m$ and, thus, convex interior $\mathring{D}_{m}$. 
 Taking the corresponding Lorentzian products, the sequence $\{\Lo\times D_m\}$ is composed by globally hyperbolic spacetimes-with-timelike-boundary 
 satisfying:
{\em each $\Lo\times D_m$ has  causally simple interior $\Lo\times \mathring{D}_m$ 
but the exhausting sequence 
 converges (uniformly on compact subsets) to a non-causally simple one, $\Lo\times S$}.\footnote{This Riemannian example was also used by  Hedicke et al. \cite{Hedicke_2021} to illustrate other properties related to causal simplicity.}  

Our next example shows deals with a converging sequence as above yielding now: 
\begin{quote}
 {\em There is a causally simple spacetime $\Lo \times S$  exhausted by a sequence $\{\Lo\times D_m\}$ of globally hyperbolic spacetimes-with-timelike-boundary and causally simple interior (thus, each one with $T_2$-lightspace)   
which is not a $T_2$-lightspace.}
\end{quote}
 As depicted in Figure \ref{fig:3},  consider the plane $\mathbb{R}^2$ and, for $k=1,2\dots$,  the rectangles 
 \[R_k:=[a_k,b_k]\times [c_k,d_k]:= \left[1-\frac{1}{k}-\frac{1}{3k(k+1)},1-\frac{1}{k}+\frac{1}{3k(k+1)}\right]\times \left[-1+\frac{1}{2^k},1-\frac{1}{2^{k}}\right].\]
Let $S$ be the open subset obtained by removing the vertical segments inside $R_k$:
$$V_k:= \left( 
\{x_k=\frac{a_{k}+b_k}{2}\}\times \left[-1+\frac{1}{2^{k +  1}},1-\frac{1}{2^{k + 1}}\right]
\right), $$
that is, $S=\left((-\infty,1)\times \mathbb{R} \right)  \setminus (\cup_k V_k)$. Finally, consider the Riemannian metric $g_S$ on $S$ which is equal to the usual one  outside all the rectangles and making each rectangle
$R_k\setminus V_k$ 
complete (with the distance induced by $g_S$ restricted to it).  

To check the required properties of $(S,g_S)$, note that, for each $l_m=(b_m+a_{m+1})/2$,  the region $D_m:=\{x \leq l_m\}$ is complete and has convex boundary. Therefore, for each two point $z_1, z_2\in  D_m$ at each homotopy class of curves in $D_m$ connecting $z_{1}$ and $z_{2}$, there is a geodesic  which minimizes the length in that class.   In particular, $\R\times D_m$ becomes globally hyperbolic and its interior is causally simple with $T_2$-lightspace. 

Moreover, $S$ is convex because, once $z_1, z_2$ are prescribed, one can take $m$ equal to the minimum integer such that $z_1, z_2\in \mathring{D}_m$ and, then,  the geodesic minimizing in $D_m$ also minimizes in $S$. Thus, $\R\times S$ is causally simple. 

It remains  to check the non-Hausdorffness of $\R\times S$, consider the points $z_\pm=(0,\pm 1)$ and the ``horizontal'' inextensible  geodesics $\sigma_\pm(s)=(s,\pm 1 ), s\in [0,1)$.
For each $m$, consider the homotopy 
class $H_m$ of  curves from $z_-$ to $z_+$ contained in $D_{m+1}$ generated by the curve $c_m$  
obtained as follows:  concatenate 
$\sigma_-|_{[0,l_m]}$ with the vertical curve $s\mapsto (l_m,s), s\in [-1,1]$ and with the horizontal one $\sigma_+(l_m-s)$, $s\in [0,l_m]$.

Let $\sigma_m$ be the geodesic connecting $z_-$ and $z_+$
which minimizes the length in the class $H_m$.  As $g$ is flat outside the rectangles $R_k$, each $\sigma_k$ must be a straight line on each interval where it lies outside $R_k$. 
Thus, necessarily the directions of the sequence 
of velocities $\{\sigma'_k\}$ at $z_\pm$ must converge to those of $\{\sigma'_\pm\}$. 
Therefore, the sequence $\{\sigma_m\}$ converges in $S$ to both geodesics, $\sigma_-$ and $\sigma_+$.  

Once this property is established for the $\sigma_m$'s in $S$, consider their lightlike liftings $\{\gamma_m\}$ starting at $(0, z_-)$ and arriving at points $(T_m, z_+)$ with $1 < T_m < 4$. This number is the sum of the lengths of $\sigma_-$, $\sigma_+$, and any vertical segment $s \mapsto (l_m, s)$ connecting them. Thus, 4 is an upper bound for the lengths of all the curves $c_m$ generating the homotopy classes $H_m$.
. This bound implies that, up to a subsequence, the endpoints of $\gamma_m$ converge to a point $(T_\infty, z_+)$. Therefore, the sequence of lightlike geodesics  $\{\gamma_m\}$ converges to both the lightlike lifting of $\sigma_-$ starting at $(0,z_-)$ and the one of $\sigma_+$ starting at $(T_\infty, z_+)$, as required.

  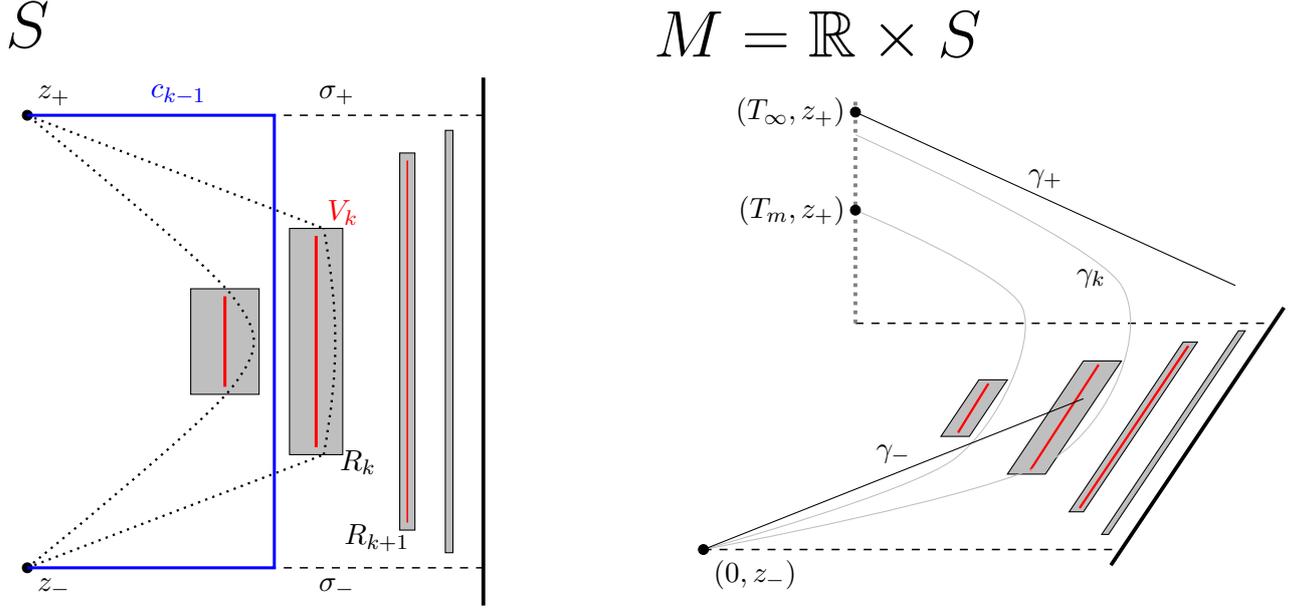
\begin{figure}
    \centering
    \begin{minipage}{0.4\paperwidth}
      \begin{tikzpicture}
        \draw [line width = 0.2mm, dashed] (-6,-3) -- node[below right, xshift=0.7cm] {$\sigma_-$} (0,-3);
        \draw [line width = 0.2mm, dashed] (-6,3) -- node[above right,, xshift=0.7cm] {$\sigma_+$} (0,3);
        \draw [line width = 0.5mm] (0,-3.5) -- (0,3.5);
        \filldraw[fill=lightgray] (-2.95,-0.7) rectangle (-3.85,0.7);
        \draw [line width = 0.4mm, color=red] (-3.4,-0.6) -- (-3.4,0.6);
        \filldraw[fill=lightgray] (-1.85,-1.5) rectangle (-2.55,1.5);
           \draw [line width = 0.4mm, color=red] (-2.2,-1.4) -- node[above right, yshift=1.4cm, color=red] {$V_k$} (-2.2,1.4);
        \filldraw[fill=lightgray] (-1.1,-2.5) rectangle (-0.9,2.5);
           \draw [line width = 0.2mm, , color=red] (-1.0,-2.4) -- (-1.0,2.4);
        \filldraw[fill=lightgray] (-0.4,-2.8) rectangle (-0.5,2.8);
        
        \filldraw[fill=black] (-6,-3) circle (0.65mm)node[below right] {$z_{-}$};
        \filldraw[fill=black] (-6,3) circle (0.65mm) node[above right] {$z_{+}$};
        \draw [dotted, line width = 0.3mm] plot [smooth] coordinates{(-6,-3) (-3.3,-0.6) (-3.3,0.6) (-6,3)};
        \draw[dotted, line width = 0.3mm] (-6,-3) -- (-2.1,-1.5); 
                \draw[dotted, line width = 0.3mm] (-6,3) -- (-2.1,1.5); 
                 \path [black, dotted, line width = 0.3mm, out=80, in=-80] (-2.1,-1.5) edge (-2.1,1.5);
        \draw [line width = 0.4mm, color = blue, mylabel= at 0.8 above left with {$c_{k-1}$}] (-6,-3) -- (-2.75,-3) -- (-2.75,3) -- (-6,3);
        \node at (-1.65,-1.6) {$R_k$};
        \node at (-1.4,-2.6) {$R_{k+1}$};
        \node at (-6,4.2) {\Huge $S$};
      \end{tikzpicture}
    \end{minipage}
        \begin{minipage}{0.3\paperwidth}
      \begin{tikzpicture}
        \draw [line width = 0.2mm, dashed] (-6.5,-2) --  (-1,-2);
        \draw [line width = 0.2mm, dashed] (-4.5,1) -- (1,1);
        \draw [line width = 0.5mm, shorten >= -2.5mm, shorten <= -2.5mm] (-1,-2) -- (1,1);
        \draw [line width = 0.5mm, color=gray, dotted] (-4.5,1) -- (-4.5,4);
        \filldraw[fill=lightgray, scale=0.25, xshift=-10cm] (-3.5,-2) -- (-2,-2) -- (0,1) -- (-1.5,1) -- cycle;
        \draw [line width = 0.3mm, color=red] (-3.15,-0.45) -- (-2.75,0.2);
        \filldraw[fill=lightgray, scale=0.5, xshift=-1.5cm] (-3.5,-2) -- (-2.5,-2) -- (-0.5,1) -- (-1.5,1) -- cycle;
         \draw [line width = 0.3mm, color=red] (-2.2,-0.950) -- (-1.3,0.45);
        \filldraw[fill=lightgray, scale=0.75, xshift=1.25cm] (-3.5,-2) -- (-3.25,-2) -- (-1.25,1) -- (-1.5,1) -- cycle;
                 \draw [line width = 0.3mm, color=red] (-1.55,-1.45) -- (-0.12,0.7);
        \filldraw[fill=lightgray, scale=0.9, xshift=-0.1cm] (-1.3,-2) -- (-1.2,-2) -- (0.8,1) -- (0.7,1) -- cycle;
        \draw[line width = 0.4pt, color=lightgray] plot[smooth] coordinates{(-6.5,-2) (-3.2,-0.8) (-2.3,1.2) (-4.5,2.5)} node [above] {};
        \draw [line width = 0.4pt, color=lightgray] plot [smooth] coordinates{(-6.5,-2) (-1.9,-0.8) (-1.,1.5) (-4.5,3.5)};
        \draw (-6.5,-2) -- node [above] {$\gamma_{-}$} (-1.5,0);
        \draw (-4.5,3.8) -- node [above] {$\gamma_{+}$} (0.5,1.5);
        \filldraw[fill=black] (-6.5,-2) circle (0.65mm)node[below right] {$(0,z_{-})$};
        \filldraw[fill=black] (-4.5,2.5) circle (0.65mm) node[left] {$(T_m,z_+)$};
        \filldraw[fill=black] (-4.5,3.8) circle (0.65mm) node[left] {$(T_\infty,z_+)$};
        
        \node at (-1.4,1.6) {$\gamma_k$};
        \node at (-5.,4.8) {\Huge$M=\mathbb{R}\times S$};

      \end{tikzpicture}
    \end{minipage}
    \caption{On the left, the  convex  Riemannian manifold $(S,g_{S})$. The rectangles $R_k$ (in gray) are removed. Some of the geodesics $\sigma_k$ plus $\sigma_{\pm}$ are depicted for visualization. 
    On the right, the corresponding Lorentzian product $M=\Lo\times S$ (the vertical products 
    on $R_k$ are not depicted for a better visualization). The geodesics $\sigma_k$ are lifted into lightlike pregeodesics $\gamma_k$ converging to both, $\gamma_{\pm}$ (also depicted in the figure). }
    \label{fig:3}
  \end{figure}
  
  

\section{Appendix: Anisotropic connections}\label{A_anisotropic}
There are three levels of structures associated with a semi-Finsler metric $L$ relevant for us, each one with increasing complexity: geodesic spray  / nonlinear connection $N$ / anisotropic connection $\nabla$ (see \cite{Javaloyes_2022, Merida}). As seen in Definition \ref{def:main2:13}, $L$ produces such a spray $G$, and it is well known that this spray selects the Berwald nonlinear connection. However, the concept of anisotropic connection is not as well spread in the literature. We used it for an easy comparison with the Levi-Civita connection of a semi-Riemannian manifold. 
Next, our final aim is to introduce the notion of {\em anisotropic connection}  (stressing that it can be understood as a direction dependent affine connection) and to see that $G$, and thus $L$ determines naturally the   {\em Berwald} and the {\em Levi-Civita-Chern} anisotropic connections.
Both share the same geodesics and any of them can be used in Section \ref{s2} (being the first one simpler). However, the Levi-Civita-Chern is more convenient for the proof of Theorem \ref{thm:main2:4}.


The practical rule to define locally the {\em Berwald anisotropic connection} from the formal Christoffel symbols in \eqref{eq:64} is as follows. For each $A$-admissible vector  field $V\in \mathfrak{X}(\withboundaryM)$ (as in Definition~\ref{d_hess})  
 define the affine connection $\nabla^V$:
	\begin{equation}
	\label{e_chris}
	\nabla^V_{\partial_i} {\partial_j}= \Gamma^k_{ij}(x,V(x)) \partial_k. 
\end{equation}
If $\bar V$ is also $A$-admissible, the dependence of  $\Gamma^k_{ij}(x,V(x))$ with $y=V(x)$ implies 
\begin{equation}
	\label{e_anisV}
	\bar V(x_0)= V(x_0) \quad \hbox{at some $x_0\in U$} \qquad \Longrightarrow \qquad \nabla^{\bar V}_{\partial_i} {\partial_j}|_{x_0}=\nabla^V_{\partial_i} {\partial_j}|_{x_0}
\end{equation} 
and the positive $0$-homogeneity of $\Gamma^k_{ij}$  $\nabla^V=\nabla^{\lambda V}$ for $\lambda >0$.
The rule of transformation of the formal Christoffel symbols implies that \eqref{e_chris} is independent of   coordinates. In addition, it matches the following global intrinsic definition of anisotropic connection,  carefully developed in \cite{Javaloyes_2022}.
\begin{definition}
	\label{def:main2:6}
	Let $(\withboundaryM,L,{\withboundaryA})$ be a semi-Finsler manifold with  boundary. An \emph{anisotropic connection} is a map
	\begin{equation}
		\label{eq:23}
		\nablaM: \mathfrak{X}(\withboundaryM)\times \mathfrak{X}(\withboundaryM)\rightarrow \mathfrak{X}(\withboundaryM_{{\withboundaryA}}),\qquad (X,Y)\rightarrow \nablaM_XY,
	\end{equation} satisfying:
	\begin{enumerate}[label=\arabic*)]
		\item $\nablaM_X \left( Y+Z \right)= \nablaM_X Y + \nablaM_XZ$ for any $X,Y,Z\in \mathfrak{X}(\withboundaryM)$,
		\item $\nablaM_X(f Y) = (X(f)Y)\circ \pi_{\withboundaryA} + (f\circ \pi_{\withboundaryA})\nablaM_X Y$ for any $f\in \mathcal{F}(\withboundaryM)$ and $X,Y\in \mathfrak{X}(\withboundaryM)$,
		\item $\nablaM_{fX+hY}Z=\left( f\circ \pi_{\withboundaryA} \right)\nablaM_XZ + (h\circ \pi_{\withboundaryA})\nablaM_YZ$ for any $f,h\in \mathcal{F}(\withboundaryM)$ and $X,Y,Z\in\mathfrak{X}(\withboundaryM)$,
	\end{enumerate}
	being $\mathfrak{X}(\withboundaryM_{{\withboundaryA}})$  the space of \emph{anisotropic vector fields}, that is, the space of sections $\withboundaryA\rightarrow \pi_{\withboundaryA}^{*}(T(M))$, where $\pi_{\withboundaryA}=\pi|_{\withboundaryA}$ is the restriction to $\withboundaryA$ of $\pi:T\withboundaryM\rightarrow \withboundaryM$  
	and $\pi_{\withboundaryA}^*$ denotes the pullback bundle. 
	
	The unique anisotropic connection which is given in local coordinates by \eqref{e_chris} with formal Christoffel symbols in \eqref{eq:64} is called the {\em Berwald anisotropic connection} of $L$.
\end{definition}

The notation $\nablaM_X^vY=\left( \nablaM_XY \right)_v$  stresses \eqref{e_anisV} and, consistently, one has $\nablaM^V_XY:=\left( \nablaM_XY \right)\circ V$, for any $A$-admissible $V$. 

The Levi-Civita-Chern anisotropic connection $\nablaLCC$ is characterized because it is  torsionless and satisfies a Koszul-type formula in terms of the so-called {\em Cartan tensor} (see \cite[Theorem 4]{Javaloyes_2022}). Its Christoffel  symbols $\hat \Gamma^k_{ij}$ can be computed from the expression of the formal ones $\Gamma^k_{ij}$ in \eqref{eq:64} just replacing the partial  derivatives $\partial_l$ by the {\em horizontal} ones:
$$
\left. \frac{\partial}{ \delta x^l}\right|_{(x,y)} := \left. \frac{\partial}{\partial x^l}\right|_{(x,y)} - N_l^k(x,y) \left. \frac{\partial}{\partial y^k}\right|_{(x,y)}
$$
where $N_l^k(x,y)= \Gamma^k_{lj}(x,y) y^j$ are the components of the nonlinear connection. 

\section*{Acknowledgments}
J.H. is partially supported by the Spanish MCIN Grant PID2021-126217NB-
I00 with FEDER funds. M.S. is partially supported by the project PID2020-116126GB-I00 funded by MCIN/ AEI /10.13039/501100011033 and
by the framework of IMAG-María de Maeztu grant CEX2020-001105-M funded by MCIN/AEI/10.13039/50110001103.	

\subsection*{Conflict of Interest Statement}
The authors have no competing interests to declare that are relevant to the content of this article.

\subsection*{Data Availability Statement}
No datasets were generated or analysed during the current study.


\begin{thebibliography}{}
  
\bibitem{AaJa} 
{\sc A.B. Aazami, and M.A. Javaloyes}, {\em Penrose's singularity theorem in a Finsler spacetime}, Classical and Quantum Gravity, 33 (2015), 025003.

\bibitem{AkeSpacetimecoveringscasual2017}
{\sc L.~A. Ak{\'e} and J.~Herrera}, {\em Spacetime coverings and the causal
 boundary}, Journal of High Energy Physics, 2017 (2017), p.~51.

\bibitem{RMI_AFS}
{\sc L.~A. Ak{\'e}, J.~L. Flores~Dorado, and M.~Sánchez~Caja}, {\em Structure of
  globally hyperbolic spacetimes-with-timelike-boundary}, Revista Matemática
  Iberoamericana, 37 (2021), p.~46–94.

\bibitem{BGS}
{\sc R.~Bartolo, A.~V. Germinario, and M.~Sánchez}, {\em Convexity of Domains of Riemannian Manifolds}, 
Annals of Global Analysis and Geometry. 21 (2002) p.~63–84.
  
\bibitem{Bartolo_2010}
{\sc R.~Bartolo, E.~Caponio, A.~V. Germinario, and M.~Sánchez}, {\em Convex
  domains of Finsler and Riemannian manifolds}, Calculus of Variations and
  Partial Differential Equations, 40 (2010), p.~335–356.

\bibitem{Bautista_2014}
{\sc A.~Bautista, A.~Ibort, and J.~Lafuente}, {\em On the space of light rays
  of a spacetime and a reconstruction theorem by low}, Classical and Quantum
  Gravity, 31 (2014), p.~075020.

\bibitem{Bautista_2022}
{\sc A.~Bautista, A.~Ibort, and J.~Lafuente}, {\em The space of light rays:
  Causality and l–boundary}, General Relativity and Gravitation, 54 (2022).

\bibitem{BEE}
{\sc J.~K. Beem, P.~Ehrlich, and K.~Easley}, {\em {Global Lorentzian
  Geometry}}, vol.~202 of Pure and Applied Mathematics, Marcel Dekker, New
  York, 1996.

\bibitem{BS03}
{\sc A.~N. Bernal and M.~Sánchez}, {\em On smooth {C}auchy hypersurfaces and
  {G}eroch’s splitting theorem}, Communications in Mathematical Physics, 243
  (2003), p.~461–470.

\bibitem{BS05}
\leavevmode\vrule height 2pt depth -1.6pt width 23pt, {\em Smoothness of time
  functions and the metric splitting of globally hyperbolic spacetimes},
  Communications in Mathematical Physics, 257 (2005), p.~43–50.

  \bibitem{BS07}
\leavevmode\vrule height 2pt depth -1.6pt width 23pt, 
{\em Further Results on the Smoothability of Cauchy Hypersurfaces and Cauchy Time Functions},  Letters in Mathematical Physics, 77 (2006) 183–197. 


\bibitem{Bishop} {\sc R.L Bishop}, {\em Infinitesimal convexity implies local convexity}, Indiana Univ. Math. J., 24 (1974) 169–172. 



\bibitem{Caponio_2011}
{\sc E.~Caponio, M.~A. Javaloyes, and M.~Sánchez}, {\em On the interplay
  between {L}orentzian causality and {F}insler metrics of {R}anders type},
  Revista Matemática Iberoamericana, 27 (2011), p.~919–952.


\bibitem{Chernov_2018}
{\sc V.~Chernov}, {\em Conjectures on the relations of linking and causality in
  causally simple spacetimes}, Classical and Quantum Gravity, 35 (2018),
  p.~105010.

\bibitem{Chernov_2010} {\sc V.~Chernov and S. Nemirovski}, {\em Legendrian links, causality, and the Low conjecture},
Geom. Funct. Anal. 19 (2010), no. 5, 1320–1333.

\bibitem{Chernov_2020}
\leavevmode\vrule height 2pt depth -1.6pt width 23pt, 
{\em Interval topology in contact geometry},
Commun. Contemp. Math. 22 (2020), no. 5, 1950042, 19 pp.

\bibitem{Fathi2011}
{\sc A.~Fathi and A.~Siconolfi}, {\em On smooth time functions}, Mathematical
  Proceedings of the Cambridge Philosophical Society, 152 (2011), p.~303–339.

\bibitem{Fefferman-Graham-1985}
{\sc C.~Fefferman and C.~R. Graham}, {\em Conformal invariants}, in \'Elie
  Cartan et les math\'ematiques d'aujourd'hui - Lyon, 25-29 juin 1984, no.~S131.
  In: Ast\'erisque, Soci\'et\'e math\'ematique de France, 1985, pp.~95--116.

\bibitem{FHS_ATMP}
{\sc J.~L. Flores, J.~Herrera, and M.~S{\'a}nchez}, {\em On the final
  definition of the causal boundary and its relation with the conformal
  boundary}, Adv. Theor. Math. Phys., 15 (2011), pp.~991--1057.

\bibitem{FHS_Memoirs}
{\sc J.~L. Flores, J.~Herrera, and M.~S{\'a}nchez}, {\em Gromov, Cauchy and causal boundaries for Riemannian, Finslerian and Lorentzian manifolds}, Memoirs of the AMS,  226 (2013), No. 1064.

\bibitem{HawkingLargeScaleStructure1975}
{\sc S.~W. Hawking and G.~F.~R. Ellis}, {\em The {{Large Scale Structure}} of
  {{Space}}-{{Time}} ({{Cambridge Monographs}} on {{Mathematical Physics}})},
  {Cambridge University Press}, 1975.

\bibitem{Hedicke-2021-contact}
{\sc J.~Hedicke}, {\em The contact
    structure on the space of null geodesics}, Diff. Geom. Appl.  75
  (2021) 101715.
  
\bibitem{Hedicke_2021}
{\sc J.~Hedicke, E.~Minguzzi, B.~Schinnerl, R.~Steinbauer, and S.~Suhr}, {\em
  Causal simplicity and (maximal) null pseudoconvexity}, Classical and Quantum
  Gravity, 38 (2021), p.~227002.

\bibitem{Hedicke_2019}
{\sc J.~Hedicke and S.~Suhr}, {\em Conformally embedded spacetimes and the
  space of null geodesics}, Communications in Mathematical Physics, 375 (2019),
  p.~1561–1577.

\bibitem{Henneaux1985}
{\sc M.~Henneaux and C. Teitelboim}, {\em Asymptotically Anti-de Sitter Spaces}, Communications in Mathematical Physics, 98 (1985), p.~391-424.

\bibitem{Hintz_2017}
{\sc P.~Hintz and G.~Uhlmann}, {\em Reconstruction of {L}orentzian manifolds
  from boundary light observation sets}, International Mathematics Research
  Notices, 2019 (2017), p.~6949–6987.

\bibitem{Javaloyes_2019}
{\sc M.~A. Javaloyes}, {\em Anisotropic tensor calculus}, International Journal
  of Geometric Methods in Modern Physics, 16 (2019), p.~1941001.

\bibitem{JPS} {\sc M.~A. Javaloyes, E. Pend\'as Recondo and M.~Sánchez}, {\em Applications of cone structures to the anisotropic rheonomic Huygens' principle}, Nonlinear Analysis 209 (2021), 112337.

\bibitem{JPS_Snell}
\leavevmode\vrule height 2pt depth -1.6pt width 23pt,  {\em Generalized Fermat's principle and Snell's law  for  cone structures}, preprint (2025).

\bibitem{JS20}
{\sc M.~A. Javaloyes and M.~S\'anchez}, {\em On the definition and examples of
  cones and {F}insler spacetimes}, Revista de la Real Academia de Ciencias
  Exactas, Físicas y Naturales. Serie A. Matemáticas  (RACSAM) 114, 30 (2020) 46pp. https://doi.org/10.1007/s13398-019-00736-y.

  

\bibitem{Javaloyes_2022}  
{\sc M.~A. Javaloyes, M.~S\'anchez, and F.~F. Villase\~nor}, {\em Anisotropic connections and parallel transport in Finsler spacetimes}, Developments in Lorentzian Geometry. Springer Proceedings in Mathematics \& Statistics,  volume 389 (2022) 32 pp  ISBN: 978-3-031-05378-8.

\bibitem{Javaloyes_2020}
{\sc M.~A. Javaloyes and B.~Soares}, {\em Anisotropic conformal invariance of
  lightlike geodesics in pseudo-{F}insler manifolds}, Classical and Quantum
  Gravity, 38 (2020), Number 2.

\bibitem{Ruiloja}
{\sc R. Loja}, {\em Differential Geometry}. Notes available (2025/06/02) at \url{https://www.math.tecnico.ulisboa.pt/%7Emabreu/GD/RLF-notes.pdf}. 


\bibitem{Low_1989}
{\sc R.~J. Low}, {\em The geometry of the space of null geodesics}, Journal of
  Mathematical Physics, 30 (1989), p.~809–811.


\bibitem{Low_2001}
\leavevmode\vrule height 2pt depth -1.6pt width 23pt, {\em 
The space of null geodesics.}
Proceedings of the Third World Congress of Nonlinear Analysts, Part 5 (Catania, 2000)
Nonlinear Anal. 47 (2001), no. 5, 3005–3017.



\bibitem{Low_2006}
\leavevmode\vrule height 2pt depth -1.6pt width 23pt, {\em The space of null geodesics (and a new causal boundary)},
  Lecture Notes in Physics, (2006),  p.~35–50. 



\bibitem{Makh}
{\sc O.~Makhmali}, {\em Differential Geometric Aspects of Causal Structures},
SIGMA 14 (2018), 080, 50 pages.
  
\bibitem{MS}
{\sc E.~Minguzzi and M.~S{\'a}nchez}, {\em The causal hierarchy of spacetimes},
  in Recent developments in pseudo-{R}iemannian geometry, ESI Lect. Math.
  Phys., Eur. Math. Soc., Z\"urich, 2008, pp.~299--358.

\bibitem{molino}
{\sc P. Molino}, {\em Riemannian foliations}, Birkhäuser Boston, 1988.

\bibitem{oneill}
{\sc B.~O'Neill}, {\em Semi-Riemannian Geometry With Applications to
  Relativity}, 103, Volume 103 (Pure and Applied Mathematics), Academic Press,
  1983.
  
\bibitem{Sa_Penrose}
{\sc M.~S\'anchez}, {\em Globally hyperbolic spacetimes: slicings, boundaries and counterexamples}, General Relativity and Gravitation, 54 (2022), 124. 

\bibitem{Sa-BIRS}
{\sc M.~S\'anchez}, {\em On the foundations and applications of
Lorentz-Finsler geometry}, preprint. Available at arXiv:2511.04645. 

\bibitem{Merida}
{\sc M. Sánchez, F.~F. Villase\~nor}, {\em The ladder of Finsler-type objects and their variational problems on spacetimes} Springer
  International Publishing, 2025, to appear. 


\bibitem{SolisPHD}
 {\sc D.~Sol\'is}, {\em Global properties of asymptotically de Sitter and Anti
  de Sitter spacetimes}, PhD thesis, University of Miami, 2006. Available at arXiv:1803.01171. 

\bibitem{Szilasi2012}
{\sc J. Szilasi, R. Lovas and D.C. Kertész}, {\em Connections, Sprays and Finsler Structures}, World Scientific, 2012.


\end{thebibliography}
\end{document}